\documentclass[11pt]{amsart}

\usepackage{amsthm, amsfonts, amssymb, amscd, rotating}
\usepackage[pagebackref,colorlinks]{hyperref}
\usepackage{tikz-cd}
\usepackage{geometry}
\usepackage{marginnote}
\usepackage{mathtools}

\theoremstyle{definition}
\newtheorem{ntn}{Notation}[section]

\theoremstyle{plain}
\newtheorem{lem}[ntn]{Lemma}
\newtheorem{prp}[ntn]{Proposition}
\newtheorem{thm}[ntn]{Theorem}
\newtheorem{cor}[ntn]{Corollary}

\newtheorem{rem}[ntn]{Remark}
\newtheorem{exm}[ntn]{Example}

\numberwithin{equation}{section}

\newcommand{\N}{\mathbb{N}}
\newcommand{\z}{\mathbb{Z}}
\newcommand{\q}{\mathbb{Q}}

\newcommand{\F}{\mathbb{F}}

\newcommand{\BB}{\mathcal{B}}
\newcommand{\GG}{\mathcal{G}}
\newcommand{\RR}{\mathcal{R}}
\newcommand{\WW}{\mathcal{W}}
\newcommand{\II}{\mathcal{I}}
\newcommand{\PP}{\mathcal{P}}
\newcommand{\RP}{\mathcal{RP}}
\newcommand{\RB}{\mathcal{RB}}
\renewcommand{\SS}{\mathcal{S}}

\newcommand{\mmm}{\mathfrak{m}}

\renewcommand{\aa}{{A^\times}}
\newcommand{\tors}{{{\rm Tor}_1^{\z}}}

\newcommand{\half}{{\left[\frac{1}{2}\right]}}

\newcommand{\mth}[1]{\left[ \frac{1}{#1} \right]}

\newcommand{\lan}{\langle}
\newcommand{\ran}{\rangle}

\newcommand{\arr}{\rightarrow}
\newcommand{\larr}{\longrightarrow}

\newcommand{\Lan}{\langle \! \langle}
\newcommand{\Ran}{\rangle \! \rangle}

\newcommand{\GE}{{\rm GE}}

\newcommand{\Eb}{\mathbb{E}}
\newcommand{\Ee}{\mathrm{E}}
\newcommand{\Ind}{{\rm Ind}}

\newcommand{\spcm}{{\rm Specm}}
\newcommand{\ffrac}{{\rm Frac}}
\newcommand{\GL}{\mathit{{\rm GL}}}
\newcommand{\SL}{\mathit{{\rm SL}}}
\newcommand{\GM}{\mathit{{\rm GM}}}
\newcommand{\SM}{\mathit{{\rm SM}}}

\newcommand{\id}{{\rm id}}
\newcommand{\im}{{\rm im}}
\newcommand{\inc}{{\rm inc}}
\newcommand{\ind}{{\rm ind}}

\renewcommand{\char}{{\rm char}}
\newcommand{\coker}{{\rm coker}}
\newcommand{\alg}{{\rm\bf alg}}

\newcommand {\mtx}[2]
{\left(\!\!\!
\begin{array}{cc}
#1 \\
#2 
\end{array}
\!\!\!\right)}

\newcommand{\mtxx}[4]
{\left(\!\!
\begin{array}{cc}
\!\!#1 & \!\!#2\\
\!\!#3 & \!\!#4
\end{array}\!\!
\right)}

\begin{document}
\title{Remarks on the refined scissors congruence group and the third homology of \texorpdfstring{$\SL_2$}{SL2}}
\author{Elvis Torres P\'erez}
\address{\sf 
Department of Sciences, Pontifical Catholic University of Peru (PUCP), 
Lima, Peru}
\email{etorresp@pucp.edu.pe}

\begin{abstract}
This work revisits the key sufficient conditions on a ring $A$ that guarantee the existence of the exact sequence
\[
    H_3(\SM_2(A),\z)\to H_3(\SL_2(A),\z)\to\RB(A)\to 0
\]
and of the isomorphism 
\[
    H_3(\SL_2(A),\SM_2(A);\z)\simeq\RP_1(A).
\]
We show that there are rings for which these results hold and where the condition $-1\in\aa^2$ fails, for example certain rings $\z\mth{m}$ and certain finite local rings. In addition, we study the special case of non-dyadic local fields of characteristic different from $2$, obtaining a refined Bloch-Wigner exact sequence when $-1\notin\aa^2$.
\end{abstract}
\maketitle
\section{Introduction}

In \cite{suslin1991}, Andrei A. Suslin stated and proved a Bloch-Wigner exact sequence for any infinite field $F$: 
\begin{equation}\label{sus-BW-seq}
    0\to\tors(\mu(F),\mu(F))^\sim \to K^\ind_3(F)\to \BB(F)\to 0.    
\end{equation}

Here $\BB(F)$ is the so-called Bloch group, which (in Suslin's words) describes the non-trivial relations among the tensors $x\otimes(1-x)$ with $x\neq 0,1$. It was defined by S. Bloch in \cite{bloch2000} as a subgroup of the group $\PP(F)$, called the {\it pre-Bloch group} or the {\it scissors congruence group} of $F$; Bloch showed that it can be regarded as an approximation to $K_3^\ind(F)$. The second name of $\PP(F)$ was inspired by the work of J. L. Dupont and C. Sah on Hilbert's third problem in hyperbolic geometry \cite{dupont-sah1982}.

The group $K^\ind_3(F)$ is the {\it indecomposable part} of $K_3(F)$, and it is known that if $F$ is algebraically closed, then $K^\ind_3(F)$ coincides with the group $H_3(\SL_2(F),\z)$. The group $\tors(\mu(F),\mu(F))^\sim$ is the unique non-trivial extension of $\tors(\mu(F),\mu(F))$ by $\z/2$. 

As an intermediate step toward his Bloch-Wigner exact sequence, Suslin obtained the exact sequence
\begin{equation}\label{gm2-gl2-seq}
    H_3(\GM_2(F),\z)\to H_3(\GL_2(F),\z)\to \BB(F)\to 0, 
\end{equation}
where $\GM_2(F)$ is the group of monomial matrices
\[
    \GM_2(F):=\left\{\mtxx{a}{0}{0}{b},\mtxx{0}{c}{d}{0}:a,b,c,d\in F^\times\right\}.   
\]
 
 In 2013, with the aim of understanding the relation between $H_3(\SL_2(F),\z)$ and $K_3^\ind(F)$ over an arbitrary field, Kevin Hutchinson \cite{hutchinson-2013} defined the refined Bloch group $\RB(F)$ as a subgroup of $\RP_1(F)$, the so-called {\it refined scissors congruence group}, and constructed a refined Bloch-Wigner complex (for infinite fields)
\[
    0\to \tors(\mu(F),\mu(F))\to H_3(\SL_2(F),\z)\to\RB(F)\to 0
\]
which is exact at the outer terms and whose homology at the middle term is annihilated by $4$.
This result was later generalized to local rings with sufficiently large residue fields in \cite{hutchinson2017}. 

In \cite{B-E--2023}, while studying the above complex, B. Mirzaii and the author obtained an $\SL_2$-version of 
Suslin's exact sequence \eqref{gm2-gl2-seq}: if $A$ is a local domain with infinite residue field in which
$-1$ is a square, then we have the exact sequence
\begin{equation}\label{sm2-sl2-seq}
    H_3(\SM_2(A),\z)\to H_3(\SL_2(A),\z)\to\RB(A)\to 0
\end{equation}
where $\SM_2(A)$ denotes the group of monomial matrices with determinant $1$ (see Section \ref{sec-sseq-sm2-sl2}).

In \cite[Remark 6.7]{B-E--2023} we suggested that the condition $-1\in\aa^2$ is in fact not essential.
The present paper studies in more depth some of the conditions implying \eqref{sm2-sl2-seq}, as well as other results of \cite{B-E--2023}, such as
the homological description of $\RP_1(A)$
\begin{equation}\label{sl2-sm2-rel}
\RP_1(A)\simeq H_3(\SL_2(A),\SM_2(A);\z).
\end{equation}

The first main theorem (Theorem \ref{thm-main-1}) gives the exact sequence 
\[
    H_3(\SM_2(A),\z)\to H_3(\SL_2(A),\z)\to\frac{\RB(A)}{\II_A\psi_1(-1)}\to 0
\]
for universal $\GE_2$-rings satisfying the conditions
\begin{enumerate}
    \item $\mu_2(A)=\{\pm 1\}$,
    \item $\SS_3=0$,
\end{enumerate}
where in general $\SS_n=H_n(B(A),T(A);\z)$, together with the condition (for rings) $\im(\gamma)\subseteq\im(\alpha)$; here $\alpha$ and $\gamma$ are the maps defined in Section \ref{sec-sseq-sm2-sl2} (see diagram \eqref{d-gam-seq}). Adding the reverse inclusion $\im(\alpha)\subseteq\im(\gamma)$ yields the exact sequence \eqref{sm2-sl2-seq}. In this version the condition $\SS_2=0$ is not needed; this gives a slight generalization of \cite[Theorem 6.6]{B-E--2023} when $-1\in\aa^2$. Most of the examples presented here satisfy the equality $\im(\gamma)=\im(\alpha)$; an example satisfying the proper inclusion $\im(\alpha)\subsetneq\im(\gamma)$ is the semilocal ring $A=\z/11\times\z/13\times\z/17$. Note that $\mu_2(A)\neq\{\pm 1\}$; this seems to be the source of the failure. The existence of universal $\GE_2$-rings satisfying (1), (2) and $\im(\gamma)\subsetneq\im(\alpha)$ remains an open problem. 

In the case of domains, the two inclusions are closely related: the inclusion $\im(\gamma)\subseteq\im(\alpha)$ implies the other one, while the converse implication requires an additional condition for domains, namely $\II^2_A\otimes\mu_2(A)\subseteq\ker(\gamma)$ (see Lemma \ref{thm-main} (5)). This last condition leads to the second main theorem (Theorem \ref{prp-sm2-sl2-dom}), which gives the exact sequence \eqref{sm2-sl2-seq} for universal $\GE_2$-domains with characteristic different from 2 that satisfy
\begin{enumerate}
    \item $\SS_2$ and $\SS_3$ are trivial,
    \item $-1\notin\aa^2$,
    \item $\II^2_A\otimes\mu_2(A)\subseteq\ker\gamma$.
\end{enumerate}

Corollary \ref{sm-sl-zm} provides a family of rings $\z\mth{m}$, with $m=2^s\cdot 3^t\cdot p^{t_1}_1\cdots p^{t_r}_r$, that satisfy \eqref{sm2-sl2-seq}. Whether this family is infinite remains an open problem.

In the last section we deal with non-archimedean, non-dyadic local fields and establish the exact sequence \eqref{sm2-sl2-seq} when $-1$ is not a square. Combining this result with \cite[Theorem 6.6]{B-E--2023} (which covers the case where $-1$ is a square), we obtain that the exact sequence \eqref{sm2-sl2-seq} holds for every non-dyadic local field $F$ with $\char(F)\neq 2$. This completes the treatment of these fields. 

When $-1\notin\aa^2$, we obtain more: by determining the complete structure of $H_3(\SM_2(F),\z)$, we obtain, as our third main theorem (Theorem \ref{thm-sm2-sl2-local}), the refined Bloch-Wigner exact sequence
\[
    0\to\mu(F)^\sim\to H_3(\SL_2(F),\z)\to \RB(F)\to 0.
\]

Moreover, in our fourth main theorem (Theorem \ref{thm-sm2-sl2-g}), we extend this reasoning to universal $\GE_2$-domains with $|\GG_A|= 4$ in which $-1$ is not a square, for example $\z\half$.

{\bf Notation.} In this article all rings are commutative (with the possible exception of group rings) and have a unit element $1$. For a ring $A$, $\aa$ denotes the group of invertible elements of $A$. 
Let $\WW_A$ denote the set of all $a\in \aa$ such that $1-a\in \aa$; explicitly,
\[
    \WW_A:=\{a\in A: a(1-a)\in \aa\}.
\]
Let $\GG_A:=\aa/(\aa)^2$ and $\RR_A:=\z[\GG_A]$. The element of $\GG_A$ represented by 
$a \in \aa$ is denoted by $\lan a \ran$. We set $\Lan a\Ran:=\lan a\ran -1\in \RR_A$. Finally, let $\epsilon:\RR_A\to\z$ be the homomorphism defined by $\lan a\ran\mapsto1$ for every $a\in\aa$, and let $\II_A$ denote its kernel.

\section{\texorpdfstring{$\GE_2$}{GE2}-rings, unimodular vectors and scissors congruence}
Let $A$ be a ring. The matrices
\[
    E(a):=\mtxx{a}{1}{-1}{0}
\]
generate a subgroup of $\SL_2(A)$ which coincides with the subgroup $\Ee_2(A)$ generated by the elementary matrices
\[
    E_{12}(a):=\mtxx{1}{a}{0}{1},\quad \quad E_{21}(a):=\mtxx{1}{0}{a}{1}, 
\]
as one sees from the formulas
\[
    E_{12}(a)=E(-a)E(0)^{-1}, \quad E_{21}(a)=E(0)^{-1}E(a).
\]

The subgroup $\Ee_2(A)$, together with the subgroup $D_2$ generated by the diagonal matrices in $\GL_2(A)$, generates the subgroup $\GE_2(A)\subseteq \GL_2(A)$. It is known that if $A$ is a field or a Euclidean domain, then $\GE_2(A)=\GL_2(A)$; rings with this property are called $\GE_2$-rings. It is not difficult to show that, for any ring, $\Ee_2(A)=\SL_2(A)$ if and only if $A$ is a $\GE_2$-ring. In \cite{cohn1966}, Cohn lists the following relations among the matrices $E(a)$:
\begin{enumerate}
    \item For any $u,v\in\aa$, $D(uv)=D(u)D(v)$, where $D(a)=E(-a)E(-a^{-1})E(-a)$ for $a\in\aa$.
    \item For any $a,b\in A$, $E(a)E(0)E(b)=-D(-1)E(a+b)$.
    \item For any $u\in\aa$ and $a\in A$, $D(u)E(a)D(u)=E(u^2a)$.
\end{enumerate}

He also observed that for many rings these relations are {\it universal}, meaning that every relation among the matrices of $\Ee_2(A)$ is a consequence of Cohn's relations; such rings are called {\it universal for $\GE_2$}. A $\GE_2$-ring which is universal for $\GE_2$ is called a {\it universal $\GE_2$-ring}; see \cite[\S 4]{hutchinson2022} for a more explicit definition of these notions.

Let $A$ be a ring. A (column) vector $\pmb{u}=\mtx{u_1}{u_2}\in A^2$ is said to be \textit{unimodular} if there exists a vector $\pmb{v}=\mtx{v_1}{v_2}$ such that the matrix $[\pmb{u},\pmb{v}]:=\mtxx{u_1}{v_1}{u_2}{v_2}$ lies in $\GL_2(A)$. Consider the free abelian group $X_k(A^2)$ generated by the $\GL_2(A)$-set (with the diagonal action)
\[
\{(\lan\pmb{v_0}\ran,\lan\pmb{v_1}\ran,\cdots,\lan\pmb{v_k}\ran): \pmb{v_i} \text{ is unimodular and } [\pmb{v_i},\pmb{v_j}]\in \GL_2(A) \text{ for } i\neq j\},
\]
where $\lan \pmb{v}\ran$ denotes the line generated by the vector $\pmb{v}$. Among these lines we single out $\pmb{\infty}=\left\lan\mtx{1}{0}\right\ran$ and $\pmb{0}=\left\lan\mtx{0}{1}\right\ran$, as well as the line $\pmb{a}=\left\lan\mtx{1}{a}\right\ran$ for $a\neq 0$. In this way $X_k(A^2)$ is a left $\GL_2(A)$-module (and hence an $\SL_2(A)$-module), which can be turned into a right module by setting $m\cdot g:=g^{-1}\cdot m$ for $m\in X_k(A^2)$ and $g\in\GL_2(A)$.

For $k\geq 1$, consider the $k$-th differential
\[
\partial_k:X_k(A^2)\to X_{k-1}(A^2)
\]
defined by
\[
\partial_k((\lan\pmb{v_0}\ran,\lan\pmb{v_1}\ran,\cdots,\lan\pmb{v_k}\ran))=\sum^k_{i=0}(-1)^i(\lan\pmb{v_0}\ran,\lan\pmb{v_1}\ran,\cdots,\hat{\lan\pmb{v_i}\ran},\cdots,\lan\pmb{v_k}\ran),
\]
where $\hat{\lan\pmb{v_i}\ran}$ indicates that the $i$-th component is omitted. Let $\partial_0:X_0(A^2)\to\z$ be defined by $\lan\pmb{v_0}\ran\mapsto 1$. We thus obtain the complex
\[
    \begin{tikzcd}
    X_{\bullet}(A^2)\to\z:& \cdots \ar[r] & X_2(A^2)\ar[r,"\partial_2"] & X_1(A^2)\ar[r,"\partial_1"] & X_0(A^2) \ar[r,"\partial_0"] & \z \ar[r] & 0.
    \end{tikzcd}
\]
We say that the above complex is \textit{exact in dimension} $<k$ if the complex
\[
\begin{tikzcd}
    X_k(A^2)\ar[r,"\partial_k"] & X_{k-1}(A^2)\ar[r,"\partial_{k-1}"] & \cdots \ar[r,"\partial_2"] & X_1(A^2)\ar[r,"\partial_1"] & X_0(A^2) \ar[r,"\partial_0"] & \z \ar[r] & 0
\end{tikzcd}
\]
is exact. For instance, for local rings (or fields), the exactness of $X_{\bullet}(A^2)\to\z$ is governed by the number of elements of the residue field \cite[Lemma 3.21]{hutchinson2017}.

In \cite{hutchinson2022}, Hutchinson studies the complex of unimodular vectors as the clique complex of a graph $\Gamma(A)$ whose vertices are the unimodular rows and in which an edge joins $\lan\pmb{u}\ran$ and $\lan\pmb{v}\ran$ whenever $[\pmb{u},\pmb{v}]\in \GL_2(A)$. In this way he obtains a description of the homology groups $H_i(X_\bullet(A^2))$ for $i=0,1$. For Hutchinson's original version of these theorems see \cite[\S 7]{hutchinson2022}.
\begin{thm}[Hutchinson]
    For any commutative ring $A$, there is an isomorphism
    \[
        H_0(X_\bullet(A^2))\simeq\z[\SL_2(A)/\Ee_2(A)].
    \]
\end{thm}
\begin{proof}
    See \cite[Theorem 2.3]{B-E--2026}.
\end{proof}
Let $K_2(2,A)$ be the rank one $K_2$-group and let $C(2,A)$ be the subgroup of $K_2(2,A)$ generated by certain symbols $c(u,v)$. It is known that $A$ is universal for $\GE_2$ if and only if $K_2(2,A)$ is generated by these symbols; see \cite[Appendix A]{hutchinson2022} for a complete proof. In this setting, Hutchinson describes $H_1(X_\bullet(A^2))$ as follows.
\begin{thm}[Hutchinson]
    For any commutative ring $A$, there is an isomorphism
    \[
        H_1(X_\bullet(A^2))\simeq \Ind^{\SL_2(A)}_{\Ee_2(A)}\left(\frac{K_2(2,A)}{C(2,A)}\right)^{\textrm{ab}}.
    \]
\end{thm}
\begin{proof}
    See \cite[Theorem 2.5]{B-E--2026}.
\end{proof}
The theorems above imply that if $A$ is a universal $\GE_2$-ring, then $X_\bullet(A^2)\to\z$ is exact in dimension $<2$.
\begin{exm}
    \label{ex-ge2}
    \begin{enumerate}
        \item \cite[\S 4]{cohn1966} Local rings (in particular, fields) are universal $\GE_2$-rings.
        \item \cite[Example 4.7]{C-H2022} $\z$ is a Euclidean domain and $K_2(2,\z)$ is generated by $c(-1,-1)$; hence $\z$ is a universal $\GE_2$-ring.
        \item Let $d$ be a squarefree positive integer and let $\mathcal{O}_d$ be the ring of integers of $\q[\sqrt{-d}]$. Then $\mathcal{O}_d$ is not a $\GE_2$-ring unless $d=1,2,3,7,11$ \cite[Theorem 6.1]{cohn1966}, and it is universal for $\GE_2$ only for $d=1,3$ \cite[Remarks]{cohn1968}.
        \item \cite[Example 4.8 (Morita)]{C-H2022} The Euclidean domains $\z\left[\frac{1}{m}\right]$, where $m=p_1^{\alpha_1}\cdots p_r^{\alpha_r}$ and $(\z/p_i)^\times$ is generated by $\{-1,p_1,\dots,p_{i-1}\}$ for all $i\leq r$, are universal $\GE_2$-rings. In particular, $\z\half$, $\z\left[\frac{1}{10}\right]$ and $\z\left[\frac{1}{290}\right]$ are universal $\GE_2$-rings.
        \item Let $A$ be a semilocal ring, and let $J(A)$ be the Jacobson radical of $A$. If $A/J(A)$ does not contain $\z/2\times\z/2$ or $\z/6$ as a direct factor, then $A$ is a universal $\GE_2$-ring \cite[p. 1165]{B-E-2025-I}. In particular, $\z/m$ is a universal $\GE_2$-ring if $6\nmid m$.
    \end{enumerate}
\end{exm}

For any $k\geq 0$, consider the subgroup $Z_k(A^2):=\ker(\partial_k)$ (we shall sometimes write $X_k$ and $Z_k$ instead of $X_k(A^2)$ and $Z_k(A^2)$). Following Coronado and Hutchinson \cite{C-H2022}, we define the group $\RP(A):=H_0(\SL_2(A),Z_2(A^2))=Z_2(A^2)_{\SL_2(A)}$. The natural action of $\GL_2(A)$ on this group (by conjugation) endows $\RP(A)$ with the structure of an $\RR_A$-module (via the matrix $\mtxx{1}{0}{0}{a}$ attached to $a\in\aa$). The first few groups $X_k(A^2)$ have the following structure:
\begin{equation}\label{struct-X}
    \ X_0\simeq\Ind^{\SL_2(A)}_{B(A)}\z,\ X_1\simeq\Ind^{\SL_2(A)}_{T(A)}\z,\ X_2(A^2)=\bigoplus_{\lan a\ran\in\GG_A}\Ind_{\mu_2(A)}^{\SL_2(A)}\z.
\end{equation}
The inclusion map $Z_2(A^2)\to X_2(A^2)$ induces a map $\lambda_1:\RP(A)\to X_2(A^2)_{\SL_2(A)}$; Shapiro's lemma and the structure of $X_2(A^2)$ above give $X_2(A^2)_{\SL_2(A)}\simeq\RR_A$. The \textit{refined scissors congruence group} is then defined as the group $\RP_1(A):=\ker(\lambda_1:\RP(A)\to\RR_A)$. 

For $a\in\aa$, consider the element $\psi_1(a)\in\RP(A)$ represented by
\[
    (\pmb{\infty},\pmb{0},\pmb{a})+(\pmb{0},\pmb{\infty},\pmb{a})-(\pmb{\infty},\pmb{0},\pmb{1})-(\pmb{0},\pmb{\infty},\pmb{1}).
\]
Coronado and Hutchinson establish the following properties of this element in \cite{C-H2022}.
\begin{lem}\label{lem-psi1}
    Let $A$ be a ring. The map $\psi_1:\aa\to\RP(A)$ satisfies the following properties:
    \begin{enumerate}
        \item $\psi_1(ab)=\lan a\ran\psi_1(b)+\psi_1(a)$.
        \item $\lan -1\ran\psi_1(a)=\psi_1(a^{-1})$.
        \item $\Lan a\Ran\psi_1(b)=\Lan b\Ran\psi_1(a)$.
        \item $\psi_1(ab^2)=\psi_1(a)+\psi_1(b^2)$.
        \item $2\psi_1(-1)=0$.
        \item $\psi_1(a^2)=\Lan a\Ran\psi_1(-1)=\Lan -1\Ran\psi_1(a)$.
        \item $2\psi_1(a^2)=0$ for all $a\in\aa$. If $-1\in\aa^2$, then $\psi_1(a^2)=0$ for all $a\in\aa$.
    \end{enumerate}
\end{lem}
\begin{proof}
    See \cite[Lemma 7.8 and Lemma 7.25]{C-H2022}.
\end{proof}
Consider the truncated complex $X_\bullet(A^2)^\tau$ (abbreviated $X_\bullet^\tau$)
\[
    X_\bullet(A^2)^\tau: 0\to Z_2(A^2)\to X_2(A^2)\overset{\partial_2}{\to} X_1(A^2)\overset{\partial_1}{\to} X_0(A^2)\to 0.
\]
Tensoring it with a projective resolution $P_\bullet\to\z$ over $\SL_2(A)$ we obtain the double complex $P_\bullet\otimes_{\SL_2(A)}X_\bullet^\tau$, whose filtration by columns yields the spectral sequence
\[
    E^1_{p,q}(\SL_2(A),X^{\tau}_\bullet)=\begin{cases}
        H_q(\SL_2(A),X_p(A^2)) & p<3\\
        H_q(\SL_2(A),Z_2(A^2)) & p=3\\
        0 & p>3
    \end{cases}\Longrightarrow H_{p+q}(\SL_2(A),X^{\tau}_{\bullet}).
\]
The first page of this spectral sequence is described by Coronado and Hutchinson in \cite[Lemma 3.2]{C-H2022}; its differentials, which we denote by $d^r_{p,q}(X^\tau)$, are as follows:
\begin{equation}
    \begin{tikzcd}
        H_3(B(A),\z) & \ar[l,"(w - \id)_*"'] H_3(T(A),\z) & * & * \\
        \aa\wedge\aa\oplus\SS_2 & \ar[l,"0"'] \aa\wedge\aa & \ar[l] 0 & \ar[l] H_2(\SL_2(A),Z_2) \\
        H_1(B(A),\z) & \ar[l,"(\cdot)^{-2}"'] \aa & \ar[l,"\epsilon\otimes\inc"'] \RR_A\otimes\mu_2(A) & \ar[l,"\inc_*"'] H_1(\SL_2(A),Z_2)\\
        \z &\ar[l,"0"'] \z & \ar[l,"\epsilon"'] \RR_A & \ar[l,"\inc_*=\lambda_1"'] \RP(A) 
    \end{tikzcd}
\end{equation}
We see that $\RP(A)=E^1_{3,0}(\SL_2(A),X^{\tau})$, and since $E^2_{1,1}(\SL_2(A),X^\tau)=0$, we have $\RP_1(A)=E^2_{3,0}(\SL_2(A),X^{\tau})=E^3_{3,0}(\SL_2(A),X^{\tau})$. The \textit{refined Bloch group} of $A$ is defined as $\RB(A):=\ker(\bar{d}^3_{3,0}:\RP_1(A)\to E^3_{0,2}(\SL_2(A),X^{\tau}))$. Moreover, if the augmented complex $X^{\tau}_{\bullet}(A^2)\to\z$ (with the map $\partial_0$) is exact, then the spectral sequence above converges to $H_{p+q}(\SL_2(A),\z)$. See \cite[\S 3]{C-H2022} for further details.
\section{A spectral sequence for \texorpdfstring{$\SM_2$}{SM2} and \texorpdfstring{$\SL_2$}{SL2}}\label{sec-sseq-sm2-sl2}
Let $A$ be a ring and consider the following map of complexes:
\begin{equation}
    \label{sm-sl-cplx}
    \begin{tikzcd}
        0\ar[r] & \hat{Z}_1(A^2)\ar[r]\ar[d,"\inc"] & \hat{X}_1(A^2)\ar[r,"\hat{\partial}_1"]\ar[d,"\inc"] & \hat{X}_0(A^2)\ar[r]\ar[d,"\inc"] & 0\\
        0\ar[r] & Z_1(A^2)\ar[r] & X_1(A^2)\ar[r,"\partial_1"] & X_0(A^2)\ar[r] & 0
    \end{tikzcd}
\end{equation}
where the groups in the upper row, $\hat{X}_0(A^2)=\z\{(\pmb{\infty}),(\pmb{0})\}$, $\hat{X}_1(A^2)=\z\{(\pmb{\infty},\pmb{0}),(\pmb{0},\pmb{\infty})\}$ and $\hat{Z}_1(A^2)=\z\{(\pmb{\infty},\pmb{0})+(\pmb{0},\pmb{\infty})\}$, are modules over the group of monomial matrices
\[
\SM_2(A):=\left\{\mtxx{a}{0}{0}{a^{-1}},\mtxx{0}{b^{-1}}{-b}{0}: a,b\in\aa\right\},
\]
and the lower row is a truncated piece of the complex $X_{\bullet}(A^2)\to\z$. A direct computation shows that
\begin{equation}
    \label{struct-X-hatX}
    \hat{X}_0\simeq\Ind^{\SM_2(A)}_{T(A)}\z,\ \hat{X}_1\simeq\Ind^{\SM_2(A)}_{T(A)}\z,\  \hat{Z}_1\simeq\z
\end{equation}
where
\[
    T(A)=\left\{\mtxx{a}{0}{0}{a^{-1}}: a\in\aa\right\},\quad B(A)=\left\{\mtxx{a}{b}{0}{a^{-1}}: a\in\aa, b\in A\right\}.
\]

Note that the augmented complex $\hat{X}_{\bullet}(A^2)\to\z$ (with $\hat{\partial}_0$ defined as $\partial_0$) is exact. If, moreover, $X_{\bullet}(A^2)\to\z$ is exact in dimension $<1$ (for example, when $A$ is a $\GE_2$-ring), then the bottom row of \eqref{sm-sl-cplx}, augmented by $\partial_0$, is also exact, and tensoring \eqref{sm-sl-cplx} with a projective resolution yields the following map of convergent spectral sequences:
\begin{equation}\label{map-s-seq}
    \begin{tikzcd}
    \hat{E}^1_{p,q}\ar[r,Rightarrow]\ar[d,"\inc_*"] & H_{p+q}(\SM_2(A),\z)\ar[d,"\inc_*"]\\
    E^1_{p,q}\ar[r,Rightarrow] & H_{p+q}(\SL_2(A),\z).
\end{tikzcd}
\end{equation}

Both spectral sequences were studied in \cite[\S 1 and \S 5]{B-E--2023}; by Shapiro's lemma we obtain $\hat{E}^1_{0,q}\simeq H_q(T(A),\z)$, $\hat{E}^1_{1,q}\simeq H_q(T(A),\z)$ and $\hat{E}^1_{2,q}\simeq H_q(\SM_2(A),\z)$, and on the $E^2$-page we have $\hat{E}^2_{2,1}\simeq\GG_A$.

Furthermore, if the complex $X_{\bullet}(A^2)\to\z$ is exact in dimension $<2$ (for example, when $A$ is a universal $\GE_2$-ring), we obtain the short exact sequence $0\to Z_2(A^2)\to X_2(A^2) \overset{\partial_2}{\to} Z_1(A^2)\to 0$, which induces the long exact sequence
\begin{equation}\label{long-eseq}
H_2(\SL_2(A),Z_1) \to H_1(\SL_2(A),Z_2) \to H_1(\SL_2(A),X_2)\to H_1(\SL_2(A),Z_1)\overset{\delta}{\to} \cdots.
\end{equation}

Using the structure of $X_2(A^2)$ described in \eqref{struct-X}, together with Shapiro's lemma, we obtain $H_1(\SL_2(A),X_2(A^2))\simeq\RR_A\otimes\mu_2(A)$. Truncating the exact sequence above, we get
\[
\begin{tikzcd}
H_1(\SL_2(A),Z_2) \ar[r] & \RR_A\otimes\mu_2(A)\ar[r,"\gamma"] & E^1_{2,1} \ar[r,"\delta"] & \RP_1(A)\ar[r] & 0,
\end{tikzcd}
\]
where $\gamma$ is induced by $\partial_2$. As in \cite[\S 4]{B-E--2023}, the maps $d^1_{2,1}:E^1_{2,1}\to\mu_2(A)\subset E^1_{1,1}$ and $d^1_{2,1}\circ\gamma=\epsilon\otimes\id:\RR_A\otimes\mu_2(A)\to\mu_2(A)$ fit into a commutative diagram, and the snake lemma then yields the exact sequence
\begin{equation}
    \label{d-gam-seq}
    \begin{tikzcd}
        &\II_A\otimes\mu_2(A)\ar[r,"\gamma"] & E^2_{2,1} \ar[r,"\delta"] & \RP_1(A)\ar[r] & 0. 
    \end{tikzcd}
\end{equation}
Finally, the map of spectral sequences yields the following diagram:
\begin{equation}\label{dia-gam-alpha}
\begin{tikzcd}
& \hat{E}^2_{2,1}\simeq\GG_A\ar[d,"\alpha:=\inc_*"] & &\\
\II_A\otimes\mu_2(A)\ar[r,"\gamma"] & E^2_{2,1} \ar[r,"\delta"] & \RP_1(A)\ar[r] & 0.
\end{tikzcd}
\end{equation}
\section{The element \texorpdfstring{$\chi_1$}{X1}}
From now on, the elements $(\pmb{\infty},\pmb{0},\pmb{a})$ and $(\pmb{0},\pmb{\infty},\pmb{a})$ of $X_2(A^2)$, and $(\pmb{\infty},\pmb{0})+(\pmb{0},\pmb{\infty})$ of $X_1(A^2)$, will be denoted by $X_a$, $X'_a$ and $Y$, respectively; in bar notation, the matrices $\mtxx{0}{1}{-1}{0}\in\SM_2(A)$ and $\mtxx{a}{0}{0}{a^{-1}}$, for $a\in\aa$, will be denoted by $w$ and $a$, respectively. Finally, the maps $\alpha$, $\gamma$ and $\delta$ are those of diagram \eqref{dia-gam-alpha}, and $d^2_{2,1}:E^2_{2,1}\to H_2(B(A),\z)\simeq H_2(T(A),\z)\oplus\SS_2$ (see Section \ref{sec-hb-ht} for this decomposition) is the $E^2$-differential of the spectral sequence $E^1_{p,q}$ appearing in diagram \eqref{map-s-seq}. 

For any $a\in\aa$, consider the element $\chi_a\in E^2_{2,1}$ represented by the cycle $[wa]\otimes\partial_2(X_{-a}-X_a)$; note that $\chi_a=\lan a\ran\chi_1$.
\begin{lem}\label{alpha-gam-chi}
    The element $\chi_1\in E^2_{2,1}$ and the maps $\alpha$, $\gamma$ and $\delta$ satisfy the following properties:
    \begin{enumerate}
        \item $\delta(\chi_1)=\psi_1(-1)$.
        \item $\chi_1=[w]\otimes Y-\gamma(\lan 1\ran\otimes (-1))$.
        \item $\gamma(\Lan a\Ran\otimes (-1))-\alpha(\lan a\ran)=\Lan a\Ran\chi_1$.
        \item $(\delta\circ\alpha)(\lan a\ran)=\psi_1(a^2)=\Lan a\Ran\psi_1(-1)$.
        \item $d^2_{2,1}\circ\alpha(\lan a\ran)=(-1)\wedge a$.
        \item If $b\in\mu_2(A)$, then $\gamma(\Lan -1\Ran\otimes(b))=\alpha(\lan b\ran)$. In particular, $\gamma(\Lan-1\Ran\otimes(-1))=\alpha(\lan-1\ran)$.
        \item $\Lan -1\Ran\chi_1=0$.
        \item For any $a,b\in\aa$ we have $\Lan b\Ran\alpha(\lan a\ran)=0$, i.e. $\GG_A$ acts trivially on $\im(\alpha)$.
    \end{enumerate} 
\end{lem}
\begin{proof}
\begin{enumerate}
    \item Take a preimage of $[w]\otimes \partial_2(X_{-1}-X_1)$ under $\id\otimes\partial_2$, namely $[w]\otimes(X_{-1}-X_1)$. Applying $d_1\otimes\id$, we obtain (recall that $\psi_1(-1)$ is $2$-torsion by Lemma \ref{lem-psi1})
    \begin{align*}
        (w[]-[])\otimes(X_{-1}-X_1)&=[]\otimes(X'_1-X'_{-1}-X_{-1}+X_1)\\
        &=-\psi_1(-1)\\
        &=\psi_1(-1).
    \end{align*}
    \item 
    \begin{align*}
        [w]\otimes\partial_2(X_{-1}-X_1)&=[w]\otimes\partial_2(X_{-1}+X'_{-1}-(X'_{-1}+X_1))\\
        &=[w]\otimes Y-[w]\otimes\partial_2(wX_1+X_1)\\
        &=[w]\otimes Y-\{w[w]+[w]\}\otimes\partial_2(X_1)\\
        &=[w]\otimes Y-[-1]\otimes\partial_2(X_1)\\
        &=[w]\otimes Y-\gamma(\lan 1\ran\otimes(-1)).
    \end{align*}
    \item Clearly $\Lan a\Ran\chi_1$ is represented by $\Lan a\Ran([w]\otimes(X_{-1}-X_1))$; the identity above then implies
    \begin{align*}
        \Lan a\Ran([w]\otimes Y-[-1]\otimes\partial_2(X_1))&=[wa]\otimes Y - [w]\otimes Y-\gamma(\Lan a\Ran\otimes (-1))\\
        &=[a]\otimes Y-\gamma(\Lan a\Ran\otimes (-1))\\
        &=\alpha(\lan a\ran)-\gamma(\Lan a\Ran\otimes (-1)).
    \end{align*}
    \item Using item (1), we have
    \begin{align*}
        \Lan a\Ran\psi_1(-1)&=\Lan a\Ran\delta(\chi_1)\\
        &=\delta(\Lan a\Ran\chi_1)\\
        &=\delta(\alpha(\lan a\ran)-\gamma(\Lan a\Ran\otimes(-1))\\
        &=\delta(\alpha(\lan a\ran))\\
        &=\delta\circ\alpha(\lan a\ran).
    \end{align*}
    \item The element $\lan a\ran\in\GG_A$ is represented by $[a]\otimes Y\in B_2(\SL_2(A))\otimes_{\SL_2(A)} Z_1(A^2)$. We have
    \[
    [a]\otimes Y=d_2\otimes\id(([w|a]-[a^{-1}|w]+[a^{-1}|a])\otimes(\pmb{\infty},\pmb{0})).
    \]
    Applying $\id\otimes\partial_1$, we obtain, modulo $\im(d_3\otimes\id)$,
    \[
    \begin{split}
        d^2_{2,1}(\alpha(\lan a\ran))&=\{w([w|a]-[a^{-1}|w]+[a^{-1}|a])-([w|a]-[a^{-1}|w]+[a^{-1}|a])\}\otimes(\pmb{\infty})\\
        &=([-1|a]-[a|-1])\otimes(\pmb{\infty}),
    \end{split}
    \]
    where the last line is obtained by adding the trivial element
    \[
    d_3\otimes\id(\{-[w|w|a]+[w|a^{-1}|a]-[w|a^{-1}|a]-[a|w|w]+[a|w|a]-[a|a^{-1}|w]-[a^{-1}|a|a^{-1}]\}\otimes(\pmb{\infty})).
    \]
    \item The element $\gamma(\Lan -1\Ran\otimes(b))$ is represented by $[b]\otimes\partial_2(X_{-1}-X_1)$; hence
    \begin{align*}
        [b]\otimes\partial_2(X_{-1}-X_1)&=[b]\otimes\partial_2(X_{-1}+X'_{-1})-[b]\otimes\partial_2(wX_1+X_1)\\
        &=[b]\otimes Y - 2[b]\otimes(X_1)\\
        &=[b]\otimes Y.
    \end{align*}
    In these calculations we have used the identities $w[b]\otimes\partial_2(X_1)=[b]\otimes\partial_2(X_1)$ and $2[b]\otimes\partial_2(X_1)=0$; the first is obtained by adding the trivial element $(d_2\otimes\id)(([b|w]-[w|b])\otimes\partial_2(X_1))$, and the second by adding $(d_2\otimes\id)(-[b|b]\otimes\partial_2(X_1))$.
    \item This is a direct consequence of (3) and (6).
    \item Note that the action of any $\lan b\ran$ on the bar resolution is by conjugation by $\mtxx{1}{0}{0}{b}$, which commutes with diagonal matrices and fixes $Y=(\pmb{\infty},\pmb{0})+(\pmb{0},\pmb{\infty})\in Z_1(A^2)$. Since $\alpha(\lan a\ran)$ is represented by $[a]\otimes Y$ the claim follows.
\end{enumerate}
\end{proof}
\begin{lem}\label{d2-chi}
    The image $d^2_{2,1}(\chi_1)$ is represented by the cycle 
    \[
        (\left[\mtxx{1}{1}{0}{1}\bigg|-1\right]-\left[-1\bigg|\mtxx{1}{1}{0}{1}\right])\otimes(\pmb{\infty});   
    \]
    consequently, $d^2_{2,1}(\chi_1)\in\SS_2$.
\end{lem}
\begin{proof}
To compute this image, we chase through the diagram
\[
\begin{tikzcd}
	B_2(\SL_2(A))\otimes_{\SL_2(A)} X_0(A^ 2) & B_2(\SL_2(A))\otimes_{\SL_2(A)} X_1(A^2) 
	\ar["\id_{B_2}\otimes \partial_1"',  l] \ar[d, "d_2\otimes \id_{X_1}"] & \\
	& B_1(\SL_2(A))\otimes_{\SL_2(A)} X_1(A^ 2) & B_1(\SL_2(A))\otimes_{\SL_2(A)} 
	Z_1(A^2).\ar[l, "\id_{B_1}\otimes \inc"']
\end{tikzcd}
\]
By Lemma \ref{alpha-gam-chi}, $[w]\otimes Y-[-1]\otimes\partial_2(X_1)$ is a representative of $\chi_1$; passing through the inclusion, we obtain
\begin{align*}
	[w]\otimes Y-[-1]\otimes\partial_2(X_1)=& (w[w]+[w])\otimes (\pmb{\infty},\pmb{0})-[-1]\otimes\partial_2(X_1)\\
    =&[-1]\otimes(\pmb{\infty},\pmb{0})-[-1]\otimes\partial_2(X_1)+d_2\otimes\id([w|w]\otimes(\pmb{\infty},\pmb{0})\\
    =&[-1]\otimes(-(\pmb{0},\pmb{1})+(\pmb{\infty},\pmb{1}))+d_2\otimes\id([w|w]\otimes(\pmb{\infty},\pmb{0})\\
    =&(h_1^{-1}[-1]-g_1^{-1}[-1])\otimes(\pmb{\infty},\pmb{0})+d_2\otimes\id([w|w]\otimes(\pmb{\infty},\pmb{0})\\
	=&d_2\otimes\id(\{[h^{-1}_1|-1]-[-1|h^{-1}_1]-[g^{-1}_1|-1]+[-1|g^{-1}_1]+[w|w]\}\otimes(\pmb{\infty},\pmb{0})),
\end{align*}
where $h_1=\mtxx{1}{1}{0}{1}$ and $g_1=\mtxx{0}{1}{-1}{1}$. This lifts to the element
\begin{align*}
	\theta:&=\{[h^{-1}_1|-1]-[-1|h^{-1}_1]-[g^{-1}_1|-1]+[-1|g^{-1}_1]\}\otimes(\pmb{\infty},\pmb{0})
	\in B_2(\SL_2(A))\otimes_{\SL_2(A)} X_1(A^2).
\end{align*}
Applying $\id\otimes\partial_1$, we obtain
\begin{align*}
	(\id\otimes\partial_1)(\theta)&=\{w([h^{-1}_1|-1]-[-1|h^{-1}_1]-[g^{-1}_1|-1]+[-1|g^{-1}_1])\\
	&-[h^{-1}_1|-1]+[-1|h^{-1}_1]+[g^{-1}_1|-1]-[-1|g^{-1}_1]\}\otimes(\pmb{\infty})
\end{align*}
in $B_2(\SL_2(A))\otimes_{\SL_2(A)} X_1(A^2)$. This represents the image $d^2_{2,1}(\chi_1)\in H_2(B(A),\z)$. For $g\in\SL_2(A)$, consider the element $c(g,-1)=([g|-1]-[-1|g])\otimes(\pmb{\infty})\in H_2(B(A),\z)$. These elements satisfy the following properties:
\begin{enumerate}
	\item For $h\in B(A)$ and $g\in\SL_2(A)$, we have
	\[
		c(hg,-1)=c(h,-1)+c(g,-1).
	\]
	\item For any $g\in\SL_2(A)$,
	\[
		wc(g,-1)=c(wg,-1)-c(w,-1).
	\]
    \item $(w[w|w]-[w|w])\otimes(\pmb{\infty})=-c(w,-1)$.
    \item For $h\in B(A)$, we have $-c(h,-1)=c(h^{-1},-1)$.
\end{enumerate}

These identities imply
\begin{align*}
	d^2_{2,1}(\chi_1)&=c(wh_1^{-1},-1)-c(h_1^{-1},-1)-c(wg_1^{-1},-1)+c(g_1^{-1},-1)-c(w,-1)\in H_2(B(A),\z).
\end{align*}
Using the identities $wg_1^{-1}=h_1^{-1}wh_1^{-1}$ and $g_1^{-1}=-h_1w$, together with the properties of $c(\cdot,\cdot)$ above, we obtain
\[
d^2_{2,1}(\chi_1)=-3c(h^{-1}_1,-1)=c(h_1,-1)=c\left(\mtxx{1}{1}{0}{1},-1\right).
\]
\end{proof}
Let $\mu(A)$ denote the group of roots of unity of $A$, and for $n\in\N$ let $\mu_n(A)$ be the group of $n$-th roots of unity in $A$; when $A$ is a domain, $\mu_n(A)$ is finite cyclic. We write $\mu_{2^{\infty}}(A)$ for the $2$-primary subgroup of $\mu(A)$. The next lemma is a reformulation of \cite[Lemma 4.5]{B-E--2023}.
\begin{lem}\label{ker-d-alpha}
    Let $A$ be a domain with $\char(A)\neq 2$. If $-1\wedge a=0$ in $\aa\wedge\aa$, then $a=p^l\nu$ for some $p\in\aa$, where $\nu\in\mu_l(A)$ for some even integer $l$.
\end{lem}
\begin{proof}
    Write $\aa=\varinjlim H$, where $H$ runs over the finitely generated subgroups of $\aa$. Since the exterior square (that is, second homology) commutes with direct limits, we may regard $-1\wedge a\in \aa\wedge \aa$ as an element of $H\wedge H$ for a suitable $H$ containing $-1$ and $a$. Decompose $H=F\oplus T$ into its free part $F$ and its torsion part $T$ (the roots of unity in $H$); note that $T$ is finite cyclic and has even order since $-1\in T$. By the K\"unneth formula, we obtain
    \[
        H\wedge H=F\wedge F\oplus F\otimes T\oplus T\wedge T.
    \]
    The summand $T\wedge T$ vanishes because $T$ is finite cyclic (in fact $T=\mu_l(A)$ since $A$ is a domain). Write $a=q\nu$ with $q\in F$ and $\nu\in T$; then $a\wedge (-1)=(0,q\otimes(-1),\nu\wedge (-1))=(0,q\otimes(-1),0)$. Setting $T\simeq \z/l$ with $l=|T|$, we see that if $-1\wedge a=0$, then $\displaystyle 0=q\otimes(-1)\in F\otimes T\simeq\frac{F}{lF}$, so that $q=p^l$ for some $p\in F$, with $l$ even.
\end{proof}
\begin{lem}
    \label{ker-wed-2}
    If $A$ is a domain with $\char(A)\neq 2$, then the $2$-torsion subgroup of $\bigwedge^2\aa$ is generated by the elements $q\wedge (-1)$ with $q\in\aa$.
\end{lem}
\begin{proof}
    Let $\beta\in\bigwedge^2\aa$ with $2\beta=0$. Since $\aa$ is the direct limit of its finitely generated subgroups, we may choose such a subgroup $H$ which contains $-1$ and satisfies $\beta\in\bigwedge^2H$. Let $H=F\oplus T$ be its decomposition into free and torsion parts. By the K\"unneth formula,
    \[
        H\wedge H=F\wedge F\oplus F\otimes T\oplus T\wedge T=F\wedge F\oplus F\otimes T,
    \]
    where the summand $T\wedge T$ vanishes because $T$ is finite cyclic. Since $\beta$ is a torsion element of $H\wedge H$ and $F\wedge F$ is free, we have $\displaystyle\beta\in F\otimes T\simeq \frac{F}{|T|F}$. We may therefore write $\beta=p\wedge\zeta$, where $p\in F$ and $\zeta\in T$ is a generator of $T$. Since $2\beta=0$, we have $p^2=q^{|T|}$ for some $q$, and hence (note that $T$ has even order because $-1\in T$)
    \[
        p=q^{\frac{|T|}{2}}.
    \]
    Finally, $\beta=p\wedge\zeta=q\wedge\zeta^{|T|/2}=q\wedge(-1)$.
\end{proof}
\begin{cor}\label{ker-gam}
    Let $A$ be a domain with $\char(A)\neq 2$. Then every $\theta\in\ker(\gamma)$ is of the form $\theta=K+\Lan \nu\Ran\otimes(-1)$, where $K\in\II^2_A\otimes\mu_2(A)$ and $\nu\in\mu_l(A)$ with $l$ even.
\end{cor}
\begin{proof}
    Let $x=\displaystyle\sum_{i=1}^n\Lan a_i\Ran\otimes(-1)\in\ker(\gamma)$. Applying $\gamma$ to $x$ and using Lemma \ref{alpha-gam-chi} (3), we obtain
    \[
        0=\alpha(\lan\prod_i a_i\ran)+\sum_i\Lan a_i\Ran\chi_1.
    \]
    
    Applying the differential $d^2_{2,1}$, we obtain an element of $H_2(B(A),\z)$. Lemma \ref{alpha-gam-chi} (5), Lemma \ref{d2-chi} and the decomposition $H_2(B(A),\z)\simeq H_2(T(A),\z)\oplus \SS_2$ then imply that
    \[
        0=-1\wedge\prod_{i=1}^na_i=c(-1,\prod_{i=1}^nh_{a_i+1}).
    \]
    Hence, by Lemma \ref{ker-d-alpha}, $\displaystyle\lan\prod_{i=1}^na_i\ran=\lan\nu\ran$, where $\nu$ is an $l$-th root of unity with $l$ even. Note that
    \[
        \Lan\nu\Ran\otimes(-1)=\Lan\prod_{i=1}^na_i\Ran\otimes(-1)=\Lan a_1\Ran\Lan\prod_{i>1}^na_i\Ran\otimes(-1)+\Lan a_1\Ran\otimes(-1)+\Lan\prod_{i>1}^na_i\Ran\otimes(-1),
    \]
    and hence
    \[
        \Lan a_1\Ran\otimes(-1)+\Lan\prod_{i>1}^na_i\Ran\otimes(-1)=\Lan\nu\Ran\otimes(-1)+\Lan a_1\Ran\Lan\prod_{i>1}^na_i\Ran\otimes(-1),
    \]
    where the second summand on the right-hand side lies in $\II^2_A\otimes\mu_2(A)$. Iterating this process in order to reduce $\displaystyle\Lan\prod_{i>1}^na_i\Ran\otimes(-1)$, we obtain
    \[
        \sum_{i=1}^n\Lan a_i\Ran\otimes(-1)=K+\Lan\nu\Ran\otimes(-1),
    \]
    where $K\in \II^2_A\otimes\mu_2(A)$.
\end{proof}
\section{The element \texorpdfstring{$\theta_a$}{Oa}}
For $a\in\aa$, consider the element $\theta_a=[w]\otimes\partial_2(X_{-a^{-1}}-X_a)\in E^2_{2,1}$. Then
\begin{equation}
    \label{tht-id}
    \begin{split}
    \theta_a&=[w]\otimes\partial_2(X_{-a^{-1}}+X'_{-a^{-1}}-X'_{-a^{-1}}-X_a)\\
    &=[w]\otimes Y - [w]\otimes\partial_2(X'_{-a^{-1}}+X_a)\\
    &=[w]\otimes Y - [w]\otimes\partial_2(wX_a+X_a)\\
    &=[w]\otimes Y - \{w[w]+[w]\}\otimes\partial_2(X_a)\\
    &=[w]\otimes Y - [-1]\otimes\partial_2(X_a).
    \end{split}
\end{equation}
\begin{prp}
    \label{sqr-plus-tht}
    Let $a\in\aa$ be such that $a^2+1\in\aa$. Then $\Lan a^2+1\Ran\theta_a=0$.
\end{prp}
\begin{proof}
    For $a\in\aa$, consider the elements $\tilde{X}_a=(\pmb{a},\pmb{\infty},\pmb{0})$ and $\tilde{X}'_a=(\pmb{a},\pmb{0},\pmb{\infty})$ of $X_2(A^2)$. The element $\partial_3((\pmb{\infty},\pmb{0},\pmb{-a^{-1}},\pmb{a}))\in Z_2(A^2)$ is well defined precisely because $a^2+1\in\aa$; hence $\theta_a=[w]\otimes\partial_2((\pmb{0},\pmb{-a^{-1}},\pmb{a})-(\pmb{\infty},\pmb{-a^{-1}},\pmb{a}))$, and, setting $\rho_a=\mtxx{a}{\frac{1}{a^2+1}}{-1}{\frac{a}{a^2+1}}\in\SL_2(A)$, we have $\theta_a=[w]\otimes\rho_a\partial_2(\tilde{X}_{-a(a^2+1)}-\tilde{X}_{\frac{a^2+1}{a}})$. One verifies that $\rho_a^{-1}w\rho_a=w(a^2+1)$, and therefore
    \begin{align*}
        \theta_a&=[\rho_a^{-1}w\rho_a]\otimes\rho_a^{-1}\rho_a\partial_2(\tilde{X}_{-a(a^2+1)}-\tilde{X}_{\frac{a^2+1}{a}})\\
        &=[w(a^2+1)]\otimes\partial_2(\tilde{X}_{-a(a^2+1)}-\tilde{X}_{\frac{a^2+1}{a}})\\
        &=\lan -(a^2+1)\ran(-[-w]\otimes\partial_2(\tilde{X}_{-a^{-1}}-\tilde{X}_{a}))\\
        &=\lan -(a^2+1)\ran ([w]\otimes\partial_2(\tilde{X}_{-a^{-1}}-\tilde{X}_a))\\
        &=\lan-(a^2+1)\ran\tilde{\theta}_a,
    \end{align*}
    where $\tilde{\theta}_a:=[w]\otimes\partial_2(\tilde{X}_{-a^{-1}}-\tilde{X}_a)$. Now, using \eqref{tht-id} and the matrix $\mtxx{1}{-a^{-1}}{a}{0}\in\SL_2(A)$, which sends $X_a$ to $\tilde{X}_a$, we obtain
    \begin{align*}
        \theta_a&=[w]\otimes Y-[-1]\otimes\partial_2(X_a)\\
        &=[w]\otimes Y-[-1]\otimes\partial_2(\tilde{X}_a)\\
        &=[w]\otimes\partial_2(\tilde{X}_{-a^{-1}}+\tilde{X}'_{-a^{-1}})-[-1]\otimes\partial_2(\tilde{X}_a)\\
        &=[w]\otimes\partial_2(\tilde{X}_{-a^{-1}}+\tilde{X}'_{-a^{-1}})-\{w[w]+[w]\}\otimes\partial_2(\tilde{X}_a)\\
        &=[w]\otimes\partial_2(\tilde{X}_{-a^{-1}}+\tilde{X}'_{-a^{-1}})-[w]\otimes\partial_2(\tilde{X}'_{-a^{-1}}+\tilde{X}_a)\\
        &=[w]\otimes\partial_2(\tilde{X}_{-a^{-1}}-\tilde{X}_a)=\tilde{\theta}_a,\\
    \end{align*}
    so that $\Lan-(a^2+1)\Ran\theta_a=0$.
    Now, by Lemma \ref{alpha-gam-chi} (2) and \eqref{tht-id}, we have
    \[
    \chi_1-\theta_a=\gamma(\Lan a\Ran\otimes(-1)).
    \]
    Multiplying by $\Lan -1\Ran$ and using Lemma \ref{alpha-gam-chi} (6), (7) and (8), we get
    \begin{align*}
        -\Lan -1\Ran\theta_a&=\Lan a\Ran\gamma(\Lan -1\Ran\otimes(-1))\\
        &=\Lan a\Ran\alpha(\lan -1\ran)\\
        &=0,
    \end{align*}
    and finally
    \[
        0=\Lan -(a^2+1)\Ran\theta_a=\Lan a^2+1\Ran\Lan -1\Ran\theta_a+\Lan -1\Ran\theta_a+\Lan a^2+1\Ran\theta_a=\Lan a^2+1\Ran\theta_a.
    \]
\end{proof}
\begin{cor}
    \label{sqr-plus-chi}
    Let $a\in\aa$ be such that $a^2+1\in\aa$. Then $\Lan a^2+1\Ran\chi_1=0$.
\end{cor}
\begin{proof}
    By Lemma \ref{alpha-gam-chi} (2), we have
    \[
    \chi_a=\lan a\ran\chi_1=[wa]\otimes Y-[-1]\otimes\partial_2(X_a)=[a]\otimes Y+[w]\otimes Y-[-1]\otimes\partial_2(X_a),
    \]
    and, using \eqref{tht-id}, we obtain
    \[
        \chi_a-\theta_a=\alpha(\lan a\ran).
    \]
    Multiplying by $\lan a\ran\Lan a^2+1\Ran$ and using Proposition \ref{sqr-plus-tht}, we obtain the required identity.
\end{proof}
\begin{rem}\label{sum-2-sqr-chi}
    Corollary \ref{sqr-plus-chi} implies that if $u=x^2+y^2$, where $x$ and $y$ are units, then 
    \[
        \Lan u\Ran\chi_1=\Lan x^2+y^2\Ran\chi_1=\Lan y^2((x/y)^2+1)\Ran\chi_1=\Lan (x/y)^2+1\Ran\chi_1=0.
    \]
\end{rem}
\section{The map \texorpdfstring{$H_\bullet(B(A),\z)\to H_\bullet(T(A),\z)$}{H.(B)->H.(T)}}\label{sec-hb-ht}
Consider the group extension
\[
    \begin{tikzcd}
        1\ar[r] & N(A)\ar[r] & B(A) \ar[r] & T(A)\ar[r] & 1,
    \end{tikzcd}
\]
where the map $B(A)\to T(A)$ sends $\mtxx{a}{b}{0}{a^{-1}}$ to $\mtxx{a}{0}{0}{a^{-1}}$, and the kernel $N(A)$ consists of the matrices $\mtxx{1}{b}{0}{1}$ with $b\in A$. Clearly $N(A)\simeq A$ and $T(A)\simeq\aa$, and the resulting action of $\aa$ on $A$ (induced by conjugation) is given by $a\cdot b=a^2b$.
The projection above is split by the inclusion $T(A)\to B(A)$, which yields, for any $n\geq 0$, a canonical decomposition
\[
    H_n(B(A),\z)\simeq H_n(T(A),\z)\oplus H_n(B(A),T(A);\z),
\]
where $H_n(B(A),T(A);\z)$ denotes the relative homology of the pair $(B(A),T(A))$ (for definitions and properties, see \cite[page 153]{knudson2001}). This relative homology group will often be denoted by $\SS_n$.
\begin{lem}
    \label{ht-hb-l1}
    If $A$ is a subring of $\q$, then for any $n\geq 0$
    \[ 
        H_n(B(A),\z)\simeq H_n(T(A),\z)\oplus H_{n-1}(\aa,A),
    \]
    and if $6\in\aa$, then $H_n(B(A),\z)\simeq H_n(T(A),\z)$ for $n\geq 0$.
\end{lem}
\begin{proof}
    See \cite[Lemma 3.5]{B-E--2023}.
\end{proof}
In the case of local rings (or fields) we have the following result.
\begin{prp}\label{ht-hb-local}
    Let $A$ be a local ring with maximal ideal $\mmm_A$.
    \begin{enumerate}
        \item If $A/\mmm_A$ is infinite, or if $|A/\mmm_A|=p^d$ with $(p-1)d>2(n+1)$, then
        \[
            H_n(B(A),\z)\simeq H_n(T(A),\z).
        \]
        \item If $A$ is a domain, and if $A/\mmm_A$ is infinite or $|A/\mmm_A|=p^d$ with $(p-1)d>2n$, then
        \[
            H_n(B(A),\z)\simeq H_n(T(A),\z).
        \] 
    \end{enumerate}
\end{prp}
\begin{proof}
    See \cite[Proposition 3.19]{hutchinson2017}.
\end{proof}
More generally, we have the following.
\begin{prp}
    \label{ht-hb-semilocal}
    Let $A$ be a semilocal ring, and suppose that for every maximal ideal $\mmm\in\spcm(A)$, either $A/\mmm$ is infinite, or $|A/\mmm|=p^d$ with $(p-1)d>2(n+1)$ (respectively $(p-1)d>2n$ when $A$ is a domain). Then $H_k(T(A),\z)\simeq H_k(B(A),\z)$ for all $k\leq n$. 
\end{prp}
\begin{exm}\label{exa-hb-ht}
    \begin{enumerate}
        \item For the ring $\z$, Lemma \ref{ht-hb-l1} gives $\SS_2=H_1(\{\pm 1\},\z)=\z/2$ and $\SS_3=H_2(\{\pm 1\},\z)=0$.
        \item For a finite field $\F_q$ with $q=p^d$, if $(p-1)d>6$, then Proposition \ref{ht-hb-local} gives
        \[
            H_i(B(\F_q),\z)\simeq H_i(T(\F_q),\z)
        \]
        for $i=1,2,3$.
        \item For the finite local rings $\z/p^r$, the residue field is $\z/p$; hence, if $p>9$, Proposition \ref{ht-hb-local} gives
        \[
            H_i(B(\z/p^r),\z)\simeq H_i(T(\z/p^r),\z)
        \]
        for $i=1,2,3$.
        \item For the semilocal rings $\z/m$, where $m=\prod p_i^{t_i}$ with $p_i>9$, we have
        \[
            H_i(B(\z/m),\z)\simeq H_i(T(\z/m),\z)
        \]
        for $i=1,2,3$.
    \end{enumerate}
\end{exm}
For $m\in\N$, we denote by $A_m$ the subring $\z\left[\frac{1}{m}\right]$ of $\q$.
\begin{lem}
    Let $p$ be a prime number. Then:
    \begin{enumerate}
        \item $H_1(B(A_p),\z)\simeq H_1(T(A_p),\z)\oplus\z/(p^2-1)$.
        \item For any $n\geq 2$, $H_n(B(A_2),\z)\simeq H_n(T(A_2),\z)$.
        \item For $p\neq 2$ and $n\geq 2$, we have $H_n(B(A_p),\z)\simeq H_n(T(A_p),\z)\oplus\z/2$.
    \end{enumerate}
\end{lem}
\begin{proof}
    See \cite[Lemma 3.7]{B-E--2023}.
\end{proof}
A similar technique extends the above result to the case where $m$ has two or more prime factors; we treat the cases in which $2\in\aa$ but $3\notin\aa$.
\begin{prp}\label{prp-hb-ht-am}
    Let $m\in\N$, $r\geq 1$ and $m=2^\alpha\cdot p_1^{\alpha_1}\cdots p_r^{\alpha_r}$ with $\alpha,\alpha_i\geq 1$ and $p_i\neq 2, 3$ for $i\leq r$. Then:
    \begin{enumerate}
        \item $H_1(B(A_m),\z)\simeq H_1(T(A_m),\z)\oplus \z/3$.
        \item $H_2(B(A_m),\z)\simeq H_2(T(A_m),\z)\oplus(\z/3)^r$.
        \item $H_3(B(A_m),\z)\simeq H_3(T(A_m),\z)\oplus(\z/3)^{\frac{(r-1)r}{2}}$.
        \item If $r=1$, then $H_n(B(A_m),\z)\simeq H_n(T(A_m),\z)$ for any $n\geq 3$.
    \end{enumerate}
\end{prp}
\begin{proof}
    By Lemma \ref{ht-hb-l1}, it suffices to compute $H_{k}(A_m^\times,A_m)$ for $k=0,1,2$. Since $A_m^\times=\lan-1,2,p_1,\dots,p_r\ran$, consider the split extension
    \[
        \begin{tikzcd}
            0\to \lan 2,p_1,\dots,p_r\ran\ar[r] & A_m^\times \ar[r] & \mu_2(A_m)\ar[r] & 0, 
        \end{tikzcd}
    \]
    which yields a Lyndon/Hochschild-Serre spectral sequence
    \[
        \prescript{I}{}{E}^2_{k,l}=H_k(\lan 2,p_1,\dots,p_r\ran,H_l(\mu_2(A_m),A_m))\implies H_{k+l}(A_m^\times,A_m).
    \]
    
    Since $\mu_2(A_m)$ acts trivially on $A_m$, we have $H_0(\mu_2(A_m),A_m)=A_m$, and the $2$-divisibility of $A_m$ implies that $H_l(\mu_2(A_m),A_m)=0$ for $l\geq 1$. Hence, by the convergence of the spectral sequence, $H_k(A_m^\times,A_m)\simeq\prescript{I}{}{E}^2_{k,0}=H_k(\lan 2,p_1,\dots,p_r\ran,A_m)$.

    Consider now the split extension
    \[
        \begin{tikzcd}
            0\to \lan p_1,\dots,p_r\ran\ar[r] & \lan2,p_1,\dots,p_r\ran \ar[r] & \lan 2\ran\ar[r] & 0, 
        \end{tikzcd}
    \]
    which yields a spectral sequence
    \[
        \prescript{II}{}{E}^2_{k,l}=H_k(\lan p_1,\dots,p_r\ran,H_l(\lan 2\ran,A_m))\implies H_{k+l}(\lan2,p_1,\dots,p_r\ran,A_m).
    \]
    
    Now, $H_0(\lan 2\ran,A_m)\simeq A_m/(2^2-1)\simeq\z/3$ because $3\nmid m$. Since $\lan 2\ran$ is an infinite cyclic group and $A_m$ is torsion-free, we have $H_l(\lan 2\ran,A_m)=0$ for $l\geq 1$. Hence, by the convergence of this second spectral sequence, $H_k(A_m^\times,A_m)\simeq H_k(\lan 2,p_1,\dots,p_r\ran,A_m)\simeq\prescript{II}{}{E}^\infty_{k,0}=\prescript{II}{}{E}^2_{k,0}=H_k(\lan p_1,\dots,p_r\ran,\z/3)$.
    
    Finally, $\prescript{II}{}{E}^2_{k,0}=H_k(\lan p_1,\dots,p_r\ran,\z/3)\simeq H_k(\lan p_1,\dots,p_r\ran,\z)\otimes \z/3$, because $\lan p_1,\dots,p_r\ran$ acts trivially on $\z/3$ (here we also use the universal coefficient theorem \cite[III.1, Exercise 3]{brown1994}). For $k=0$ we obtain $\prescript{II}{}{E}^2_{0,0}=\z/3$, which proves (1). For $k\geq 1$ we distinguish two cases.
    \begin{enumerate}
        \item Suppose $r>1$. For $k=1$ we have $\prescript{II}{}{E}^2_{1,0}=\z^r\otimes(\z/3)\simeq\left(\z/3\right)^r$, and for $k=2$ we have $\prescript{II}{}{E}^2_{2,0}=\z^{\frac{(r-1)r}{2}}\otimes(\z/3)\simeq\left(\z/3\right)^{\frac{(r-1)r}{2}}$ (by the K\"unneth formula); this proves (2) and (3).
        \item Suppose $r=1$. The group $\lan p_1\ran$ is infinite cyclic, so $\prescript{II}{}{E}^2_{k,0}=0$ for $k\geq 2$, while for $k=1$ we have $\prescript{II}{}{E}^2_{1,0}=\z\otimes\z/3\simeq\z/3$; this proves (4).
    \end{enumerate}
\end{proof}
\section{The homology groups \texorpdfstring{$H_3(\SL_2(A),\z)$ and $H_3(\SL_2(A),\SM_2(A);\z)$}{H3(SL2(A),z) and H3(SL2(A),SM2(A);z)}}
Let $A$ be a ring and consider the following commutative diagram with exact rows:
\begin{equation}
    \begin{tikzcd}
        & \hat{E}^2_{2,1}\ar[r,"\alpha"]\ar[d,twoheadrightarrow,"\hat{d}^2_{2,1}"] & E^2_{2,1}\ar[r]\ar[d,"d^2_{2,1}"] & \coker(\alpha) \ar[r] \ar[d,"d'"] & 0\\
        0 \ar[r] & \{\pm 1\}\wedge\aa\ar[r] & \aa\wedge\aa\oplus\SS_2 \ar[r,twoheadrightarrow] & \frac{\aa\wedge\aa}{\aa\wedge\{\pm 1\}}\oplus\SS_2 \ar[r]& 0 
    \end{tikzcd}
\end{equation}

Here $d'$ is induced by $d^2_{2,1}$. Setting $\widetilde{\coker(\alpha)}:=\ker(d')$, the snake lemma yields the exact sequence
\begin{equation}
    \label{alph-inf}
    \begin{tikzcd}
        \hat{E}^{\infty}_{2,1}\ar[r,"\alpha"] & E^{\infty}_{2,1}\ar[r] & \widetilde{\coker(\alpha)} \ar[r] & 0.
    \end{tikzcd}
\end{equation}

The next proposition relates the cokernel of $\alpha$ to the cokernel of the map induced by the inclusion $\SM_2(A)\to \SL_2(A)$ on $H_3$. 
\begin{prp}\label{prp-sm2-sl2}
Let $A$ be a ring satisfying the following conditions:
\begin{enumerate}
    \item the complex $X_{\bullet}(A^2)\to\z$ is exact in dimension $<2$;
    \item $H_3(B(A),\z)\simeq H_3(T(A),\z)$.
\end{enumerate}
Then there is an exact sequence
\[
H_3(\SM_2(A),\z)\to H_3(\SL_2(A),\z)\to \widetilde{\coker(\alpha)}\to 0.
\]
\end{prp}
\begin{proof}
    The proof is the same as that of Theorem 6.6 in \cite{B-E--2023}. Analyzing the map \eqref{map-s-seq} of spectral sequences $\hat{E}^1_{p,q}\overset{\inc}{\longrightarrow} E^1_{p,q}$, we obtain filtrations (the equalities on the right follow from $\hat{E}^{\infty}_{3,0}=E^{\infty}_{3,0}=0$)
    \[
    \begin{tikzcd}
        0\ar[r,hook] & \hat{F}_0 \ar[r,hook]\ar[d] & \hat{F}_1 \ar[r,hook]\ar[d] & \hat{F}_2 \ar[r,equal]\ar[d] & \hat{F}_3 = H_3(\SM_2(A),\z)\ar[d]\\
        0\ar[r,hook] & F_0 \ar[r,hook] & F_1 \ar[r,hook] & F_2 \ar[r,equal] & F_3 = H_3(\SL_2(A),\z).
    \end{tikzcd}
    \]
    This yields commutative diagrams with exact rows
    \begin{equation}
    \label{rel_2-1}
    \begin{tikzcd}
        0 \ar[r] & \hat{F}_1 \ar[r]\ar[d] & H_3(\SM_2(A),\z)\ar[r]\ar[d] & \hat{E}^{\infty}_{2,1}\ar[r]\ar[d] & 0\\
        0 \ar[r] & F_1 \ar[r] & H_3(\SL_2(A),\z)\ar[r] & E^{\infty}_{2,1}\ar[r] & 0\\
    \end{tikzcd}    
    \end{equation}
    and
    \begin{equation}
    \label{rel_1-2}
    \begin{tikzcd}
        0 \ar[r] & \hat{E}^{\infty}_{0,3}=\hat{F}_0 \ar[r]\ar[d] & \hat{F}_1 \ar[r]\ar[d] & \hat{E}^{\infty}_{1,2}\ar[r]\ar[d] & 0\\
        0 \ar[r] & E^{\infty}_{0,3}=F_0 \ar[r] & F_1 \ar[r] & E^{\infty}_{1,2}\ar[r] & 0.\\
    \end{tikzcd}    
    \end{equation}
    Applying the snake lemma to diagram \eqref{rel_2-1}, we see that it suffices to prove that the map $\hat{F}_1\to F_1$ is surjective. To this end, observe that in diagram \eqref{rel_1-2} the groups $\hat{E}^{\infty}_{1,2}$ and $E^{\infty}_{1,2}$ are quotients of $H_2(T(A),\z)$, which implies that the right-hand vertical map is surjective. Moreover, the condition $H_3(B(A),\z)\simeq H_3(T(A),\z)$ implies, by the same argument, that the left-hand vertical map is surjective. Therefore $\hat{F}_1\to F_1$ is surjective.
\end{proof}
Lemma \ref{thm-main} below analyzes $\coker(\alpha)$ and provides weaker conditions under which an exact sequence as in \cite[Theorem 6.6]{B-E--2023} holds. Consider the following completion of diagram \eqref{dia-gam-alpha}:
\begin{equation}\label{dia-gam-alpha-2}
\begin{tikzcd}
& & \II_A\otimes\mu_2(A)\ar[d,"\gamma"] & & \\
& \hat{E}^2_{2,1}\simeq\GG_A \ar[r,"\alpha:=\inc_*"]\ar[d,twoheadrightarrow,"\delta'':=\delta\circ\alpha"] & E^2_{2,1} \ar[r,"\pi"]\ar[d,twoheadrightarrow,"\delta"] & \coker(\alpha)\ar[r]\ar[d,twoheadrightarrow,"\delta'"] & 0\\
0 \ar[r] & \II_A\psi_1(-1)\ar[r] & \RP_1(A)\ar[r] & \frac{\RP_1(A)}{\II_A\psi_1(-1)}\ar[r] & 0
\end{tikzcd}
\end{equation}
The surjectivity of the left-hand vertical map follows from Lemma \ref{alpha-gam-chi} (4), and the surjectivity of the right-hand vertical map follows from the five lemma.
\begin{lem}\label{gam-alpha}
    If $A$ is a domain with $\char(A)\neq 2$ such that $X_{\bullet}(A^2)\to\z$ is exact in dimension $<2$ and $\gamma(\Lan a\Ran\otimes(-1))=\alpha(\lan b\ran)$, then $\lan a\ran=\lan\nu b\ran$ where $\nu\in\mu_l(A)$ with $l$ even.
\end{lem}
\begin{proof}
    By Lemma \ref{alpha-gam-chi}(3) we have
    \[
        \alpha(\lan a\ran)+\Lan a\Ran\chi_1=\alpha(\lan b\ran)\implies \alpha(\lan ab\ran)=\Lan a\Ran\chi_1.
    \]
    Applying $d^2_{2,1}$ we have by Lemmas \ref{d2-chi} and \ref{alpha-gam-chi} (5)
    \[
        (-1)\wedge(ab)\in \SS_2\implies (-1)\wedge(ab)=0.
    \]
    Lemma \ref{ker-d-alpha} gives that $\lan ab\ran=\lan\nu\ran$ where $\nu$ is an $l$-th root of unity with $l$ even; multiplying by $\lan b\ran$ yields the claim.
\end{proof}
\begin{lem}\label{thm-main}
    Let $A$ be a ring such that $X_{\bullet}(A^2)\to\z$ is exact in dimension $<2$, and consider the maps of diagram \eqref{dia-gam-alpha-2}. Then:
    \begin{enumerate}
        \item There exists an exact sequence
        \[
            \II_A\otimes\mu_2(A)\to\coker(\alpha)\to\frac{\RP_1(A)}{\II_A\psi_1(-1)}\to 0.
        \]
        \item $\im(\gamma)\subseteq\im(\alpha)$ if and only if $\delta'$ is an isomorphism.
        \item $\im(\alpha)\subseteq\im(\gamma)$ if and only if $\II_A\psi_1(-1)=0$.
        \item If $A$ is a domain with $\char(A)\neq 2$, then $\im(\gamma)\subseteq\im(\alpha)$ implies that $\II_A\chi_1\subseteq\alpha\left(\dfrac{\mu_{2^\infty}(A)\cdot\aa^2}{\aa^2}\right)$. In particular $\II_A\chi_1\subseteq\im(\alpha)$.
        \item If $A$ is a domain with $\char(A)\neq 2$, then $\im(\gamma)\subseteq\im(\alpha)$ if and only if $\im(\alpha)\subseteq\im(\gamma)$ and $\II_A^2\otimes\mu_2(A)\subseteq\ker(\gamma)$.
    \end{enumerate}
\end{lem}
\begin{proof}
    Applying the snake lemma to diagram \eqref{dia-gam-alpha-2}, we obtain the exact sequence
    \begin{equation}\label{snk-dia-gam-alpha-2}
    \begin{tikzcd}
        \ker(\delta'')\ar[r,"\bar{\alpha}"] & \ker(\delta)=\im(\gamma)\ar[r,"\bar{\pi}"] & \ker(\delta') \ar[r] & 0.
    \end{tikzcd}
    \end{equation}
    For (1), the left-hand map is the composite $\pi\circ\gamma$ and the right-hand map is $\delta'$; exactness follows from the exact sequence above.
    
    For (2), suppose that $\im(\gamma)\subseteq\im(\alpha)$, and let $x=\gamma(y)=\alpha(z)$ with $y\in\II_A\otimes\mu_2(A)$ and $z\in\hat{E}^2_{2,1}$. Then $\delta''(z)=\delta(\alpha(z))=0$, so $\bar{\alpha}$ is surjective and $\ker(\delta')=0$. The converse follows from the exact sequence \eqref{snk-dia-gam-alpha-2}.
    \\
    For (3), it is clear that if $\im(\alpha)\subseteq\im(\gamma)$, then $\delta''=\delta\circ\alpha=0$, which implies that $\II_A\psi_1(-1)=0$. The converse follows from diagram \eqref{dia-gam-alpha-2}.
    \\
    
    For (4), let $a\in\aa$, then $\gamma(\Lan a\Ran\otimes(-1))=\alpha(\lan b\ran)$ for some $b\in\aa$. Lemma \ref{gam-alpha} gives that $\lan b\ran=\lan\nu a\ran$ where $\nu$ is an $l$-th root of unity with $l$ even; note that if $l=2^nm$ where $m$ is odd, we have $\lan\nu\ran=\lan\nu^m\ran\in\dfrac{\mu_{2^\infty}(A)\cdot\aa^2}{\aa^2}$. Lemma \ref{alpha-gam-chi} (3) gives that
    \[ 
        \gamma(\Lan a\Ran\otimes(-1))=\alpha(\lan a\ran)+\Lan a\Ran\chi_1=\alpha(\lan b\ran)=\alpha(\lan a\nu\ran).
    \]
    This implies $\Lan a\Ran\chi_1=\alpha(\lan\nu\ran)\in\im(\alpha)$.\\
    
    For (5), let $A$ be a domain and assume that $\im(\gamma)\subseteq\im(\alpha)$; by identity $\Lan a\Ran\Lan b\Ran=\Lan ab\Ran-\Lan a\Ran-\Lan b\Ran$, Lemma \ref{alpha-gam-chi} (3), (8) and part (4) we have that 
    \[
        \gamma(\Lan a\Ran\Lan b\Ran\otimes(-1))=\Lan a\Ran\Lan b\Ran\chi_1=\Lan a\Ran\alpha(\lan\nu\ran)=0.
    \]
    We shall only treat the case $-1\notin\aa^2$, since the remaining case follows from \cite[Corollary 6.4]{B-E--2023}. Let $a\in\aa$. By (4), we have that
    \[
        \alpha(\lan a\ran)=\gamma(\Lan a\Ran\otimes(-1))-\Lan a\Ran\chi_1=\gamma(\Lan a\Ran\otimes(-1))-\alpha(\lan\nu\ran).
    \]
    Since $-1\notin\aa^2$ we have $\mu_{2^\infty}(A)=\mu_2(A)$, so we may take $\nu\in\mu_2(A)$, and Lemma \ref{alpha-gam-chi} (6) yields the claim. Conversely, take an element $\sum\Lan a_i\Ran\otimes(-1)$. The hypotheses imply that, for some $b\in\aa$,
    \begin{equation}
        \label{eq-prf4}
        \alpha(\lan\prod_{i}a_i\ran)=\gamma(\Lan b\Ran\otimes(-1)).
    \end{equation}
    Lemma \ref{gam-alpha} gives that $\lan\nu\prod_{i}a_i\ran=\lan b\ran$. Now, the hypothesis $\II_A^2\otimes\mu_2(A)\subseteq\ker(\gamma)$ implies that, for any $x,y\in\aa$,
    \[
        \gamma(\Lan xy\Ran\otimes(-1))=\gamma(\Lan x\Ran\otimes(-1)+\Lan y\Ran\otimes(-1)).
    \]
    Hence, by substituting $\lan b\ran=\lan\nu\prod_{i}a_i\ran$ in equation \eqref{eq-prf4} and using the identity above we have
    \[
        \alpha(\lan\prod_{i}a_i\ran)-\gamma(\Lan\nu\Ran\otimes(-1))=\gamma(\sum_i\Lan a_i\Ran\otimes(-1)).
    \]
    And since $-1\notin\aa^2$ we consider $\nu\in\mu_2(A)=\{\pm 1\}$ as above, and Lemma \ref{alpha-gam-chi} (6) yields the claim.
\end{proof}
\begin{exm}\label{ex-im-eq}
\begin{enumerate}
    \item The Euclidean domain $\z$ is a universal $\GE_2$-ring which satisfies $\im(\gamma)\subseteq\im(\alpha)$ by Lemma \ref{alpha-gam-chi} (3) and (7). Hence, by Lemma \ref{thm-main} (5), $\im(\gamma)=\im(\alpha)$.
    \item It is known that a finite field $\F_q$ with $q=p^d$ and $p>2$ has exactly two square classes, $\lan 1\ran$ and $\lan u\ran$, where $u$ is a non-square. When $-1$ is not a square (the other case was treated in \cite{B-E--2023}) we may take $u=-1$, and then $\im(\gamma)=\im(\alpha)$ by Lemma \ref{alpha-gam-chi} (6).
    \item The rings $\z/p^r$ are local, hence universal $\GE_2$-rings; moreover, $(\z/p^r)^\times\simeq\z/p^{r-1}(p-1)$ and the square classes are $\lan 1\ran$ and $\lan u\ran$, where $u$ is a non-square. If $-1$ is not a square, then $\im(\gamma)\subseteq\im(\alpha)$ by Lemma \ref{alpha-gam-chi} (6), since $\mu_2(A)=\{\pm 1\}$. At the same time, $\II_A\psi_1(-1)=0$, and therefore $\im(\gamma)=\im(\alpha)$.
    \item Consider the domains $\z\left[\frac{1}{2}\right]$, $\z\left[\frac{1}{10}\right]$ and $\z\left[\frac{1}{290}\right]$, which are universal $\GE_2$-rings, and observe that $5=2^2+1$ and $29=5^2+2^2=2^2\{(5/2)^2+1\}$. Using the identity $\Lan ab\Ran\chi_1=\Lan a\Ran\Lan b\Ran\chi_1+\Lan a\Ran\chi_1+\Lan b\Ran\chi_1$, together with Corollary \ref{sqr-plus-chi} and Lemma \ref{alpha-gam-chi} (7), we see that every unit $a$ in these rings satisfies $\Lan a\Ran\chi_1=0$; by Lemma \ref{alpha-gam-chi} (3), this gives the inclusion $\im(\gamma)\subseteq\im(\alpha)$, and Lemma \ref{thm-main} (5) implies that $\im(\gamma)=\im(\alpha)$.
    \item The ring of Eisenstein integers $\z[\zeta]$, where $\zeta^2+\zeta+1=0$, is a universal $\GE_2$-ring. Its unit group $\z[\zeta]^\times=\{\pm 1,\pm\zeta,\pm\zeta^2\}$ is cyclic of order $6$, so that $\GG_{\z[\zeta]}=\{1,\lan-1\ran\}$. Lemma \ref{alpha-gam-chi} (3) and (7), together with Lemma \ref{thm-main} (5), imply that $\im(\gamma)=\im(\alpha)$.
    \item For the semilocal ring $A=\z/2431\simeq \z/11\times\z/13\times\z/17$, we have $\aa=\left(\z/11\right)^\times\oplus\left(\z/13\right)^\times\oplus\left(\z/17\right)^\times$. Let $e_1=2$, $e_2=2$ and $e_3=3$ be generators of the respective direct summands of $\aa$. The K\"unneth formula gives 
    \[
        \bigwedge^2\nolimits\aa\simeq\lan e_1\wedge e_2\ran\oplus\lan e_1\wedge e_3\ran\oplus\lan e_2\wedge e_3\ran,    
    \]
    where $\lan e_1\wedge e_2\ran\simeq\z/2$, $\lan e_1\wedge e_3\ran\simeq\z/2$ and $\lan e_2\wedge e_3\ran\simeq\z/4$. Let $b=1990$; then $b\equiv -1=e_1^5\mod 11$, while $b\equiv 1\mod 13$ and $b\equiv 1\mod 17$. Let $a=1431$; it satisfies $a\equiv 1\mod 11$, $a\equiv 1\mod 13$ and $a\equiv e_3\mod 17$. Note that $-1\equiv e_1^5\mod11$, $-1\equiv e_2^6\mod 13$ and $-1\equiv e_3^8\mod 17$. If $\gamma(\Lan a\Ran\otimes (b))=\alpha(\lan r\ran)$ for some unit $r$, then applying $d^2_{2,1}$ gives $a\wedge b=-1\wedge r$ in $\bigwedge^2\aa$. Writing $r\equiv e_1^{\alpha_1}\mod 11$, $r\equiv e_2^{\alpha_2}\mod 13$ and $r\equiv e_3^{\alpha_3}\mod 17$, we obtain
    \begin{align*}
        e_1\wedge e_3&=5(e_1\wedge e_3)\\&=a\wedge b\\
        &=-1\wedge r\\
        &=\alpha_2(e_1\wedge e_2)+\alpha_3(e_1\wedge e_3)+2\alpha_3(e_2\wedge e_3),
    \end{align*}
    which implies that $\alpha_3$ is both odd and even, a contradiction; therefore $\im(\gamma)\not\subseteq\im(\alpha)$. On the other hand, every class in $\GG_A$ has a representative that is a sum of two squares of units, so $\Lan u\Ran\chi_1=0$ for every unit $u$ (Remark \ref{sum-2-sqr-chi}). Hence, by Lemma \ref{alpha-gam-chi} (3), $\alpha(\lan u\ran)=\gamma(\Lan u\Ran\otimes(-1))$, and we conclude that $\im(\alpha)\subsetneq\im(\gamma)$.
\end{enumerate}
\end{exm}
The next corollary places Theorem 8.2 of \cite{B-E--2023} in the present framework.
\begin{cor}
    Let $A$ be a ring satisfying the following conditions:
    \begin{enumerate}
        \item the complex $X_{\bullet}(A^2)\to\z$ is exact in dimension $<2$;
        \item $H_i(B(A),\z)\simeq H_i(T(A),\z)$ for $i=2,3$.
    \end{enumerate}
    Then there exists an exact sequence
    \[
        \II_A\otimes\mu_2(A)\to H_3(\SL_2(A),\SM_2(A);\z)\to \frac{\RP_1(A)}{\II_A\psi_1(-1)}\to 0.
    \]
    If, in addition, $\im(\gamma)\subseteq\im(\alpha)$, then
    \[
        H_3(\SL_2(A),\SM_2(A);\z)\simeq\frac{\RP_1(A)}{\II_A\psi_1(-1)},
    \]
    and if $\im(\gamma)=\im(\alpha)$, then
    \[
        H_3(\SL_2(A),\SM_2(A);\z)\simeq\RP_1(A).
    \]
\end{cor}
\begin{proof}
    In \cite[\S 8]{B-E--2023} we constructed a spectral sequence of relative homology
    \[
        \Eb^1_{p,q}\implies H_{p+q}(\SL_2(A),\SM_2(A);\z),    
    \]
    and proved that $\Eb^2_{2,1}=\Eb^1_{2,1}=H_1(\SL_2(A),\SM_2(A);Z_1(A^2),\hat{Z}_1(A^2))$, and this group is precisely $\coker(\alpha)$. Moreover, when $\SS_2$ and $\SS_3$ are trivial, this group is $H_3(\SL_2(A),\SM_2(A);\z)$ \cite[Proposition 8.1]{B-E--2023}. Hence it suffices to apply Lemma \ref{thm-main}.
\end{proof}
\begin{rem}
    Unlike \cite[Theorem 8.2]{B-E--2023}, the corollary above does not assume $\mu_2(A)=\{\pm 1\}$: it is not used in the proof of the isomorphism
    \[
        H_1(\SL_2(A),\SM_2(A);Z_1(A^2),\hat{Z}_1(A^2))\simeq H_3(\SL_2(A),\SM_2(A);\z)
    \]
    in \cite[Proposition 8.1]{B-E--2023}. However, this condition seems to be related to the inclusion $\im(\gamma)\subseteq\im(\alpha)$; in Example \ref{ex-im-eq} (6), where $|\mu_2(A)|=8$, this inclusion fails.
\end{rem}
\begin{rem}
    In particular, if $A$ satisfies conditions (1) and (2) of the corollary above and, in addition, $\mu_2(A)=\{\pm 1\}$ and $-1\in\aa^2$, then $\im(\gamma)=\im(\alpha)$, and hence
    \[
        H_3(\SL_2(A),\SM_2(A);\z)\simeq \RP_1(A),
    \]
    which is precisely the second part of \cite[Theorem 8.2]{B-E--2023}. 
\end{rem}
\begin{exm}
    As further examples, we obtain the following.
    \begin{enumerate}
        \item For the ring $\z\half$ we have
        \[
            H_3(\SL_2(\z\mth{2}),\SM_2(\z\mth{2});\z)\simeq \RP_1(\z\mth{2}).
        \]
        \item For $\F_q$ with $q=p^d$ such that $(p-1)d>6$ we have
        \[
            H_3(\SL_2(\F_q),\SM_2(\F_q);\z)\simeq \RP_1(\F_q).
        \]
        \item For $\z/p^r$ with $p>9$ we have
        \[
            H_3(\SL_2(\z/p^r),\SM_2(\z/p^r);\z)\simeq \RP_1(\z/p^r).
        \]
        \item For semilocal rings satisfying (1) (for example, universal $\GE_2$-rings) and (2) (for example, when all residue fields are sufficiently large), we have an exact sequence
        \[
            \II_A\otimes\mu_2(A)\to H_3(\SL_2(A),\SM_2(A);\z)\to \frac{\RP_1(A)}{\II_A\psi_1(-1)}\to 0,
        \]
        and if $\im(\gamma)\subseteq\im(\alpha)$, the corollary above gives the isomorphism
        \[
            H_3(\SL_2(A),\SM_2(A);\z)\simeq \frac{\RP_1(A)}{\II_A\psi_1(-1)}.
        \]
        Examples of this are the local rings $\z/p^r$ with $p>9$. In this specific case, Example \ref{ex-im-eq} (3) gives
        \[
            H_3(\SL_2(A),\SM_2(A);\z)\simeq \RP_1(A).
        \]
    \end{enumerate}
\end{exm}
Our first main theorem establishes the exact sequence \eqref{sm2-sl2-seq} when $\im(\gamma)=\im(\alpha)$. First, note that if $\im(\gamma)\subseteq\im(\alpha)$, then Lemma \ref{alpha-gam-chi} (5) yields the commutativity of the following diagram:
\begin{equation}
    \begin{tikzcd}
        & \II_A\otimes\mu_2(A)\ar[r,"\gamma"]\ar[d,"\Delta"] & E^2_{2,1}\ar[r,"\delta"]\ar[d,"d^2_{2,1}"] & \RP_1(A)\ar[r]\ar[d,"\Delta'"] & 0\\
        0\ar[r] & \aa\wedge\mu_2(A)\ar[r] & \aa\wedge\aa\oplus \SS_2\ar[r] & \frac{\aa\wedge\aa}{\aa\wedge\mu_2(A)}\oplus \SS_2\ar[r] & 0
    \end{tikzcd}.
\end{equation}
\begin{lem}\label{ker-rba}
    If $\im(\gamma)\subseteq\im(\alpha)$, then $\RB(A)=\ker(\Delta')$.
\end{lem}
\begin{proof}
    First, $d^2_{2,1}\circ\gamma=d^2_{2,1}(X^\tau)$, and if $\im(\gamma)\subseteq\im(\alpha)$, then $E^3_{0,2}(\SL_2(A),X^\tau)=\frac{\aa\wedge\aa}{\aa\wedge\mu_2(A)}\oplus\SS_2$ by Lemma \ref{alpha-gam-chi} (5). It is then straightforward to check that $\Delta'=d^3_{2,1}(X^\tau)$ (the proof is the same as that of Lemma 2.3 in \cite{B-E-2025}).
\end{proof}
\begin{thm}
    \label{thm-main-1}
    Let $A$ be a ring satisfying the following conditions:
    \begin{enumerate}
        \item $\mu_2(A)=\{\pm 1\}$;
        \item the complex $X_{\bullet}(A^2)\to\z$ is exact in dimension $<2$;
        \item $H_3(B(A),\z)\simeq H_3(T(A),\z)$.
    \end{enumerate}

    If, in addition, $\im(\gamma)\subseteq\im(\alpha)$, then there is an exact sequence
    \[
        H_3(\SM_2(A),\z)\to H_3(\SL_2(A),\z)\to \frac{\RB(A)}{\II_A\psi_1(-1)}\to 0,
    \]
    and if $\im(\gamma)=\im(\alpha)$, then there is an exact sequence
    \[
        H_3(\SM_2(A),\z)\to H_3(\SL_2(A),\z)\to \RB(A)\to 0.
    \]
\end{thm}
\begin{proof}
Consider the following completion of diagram \eqref{dia-gam-alpha-2}:
\[
{\small
\begin{tikzcd}
   & & & \aa\wedge\mu_2(A)\ar[dd,hook] & &\\
   & & \II_A\otimes\mu_2(A)\ar[ru,"\Delta''"]\ar[dd,"\gamma" near end] & & &\\
   & \aa\wedge\{\pm 1\}\ar[rr,hook] & & \aa\wedge\aa\oplus\SS_2\ar[rr,twoheadrightarrow,"\pi''"]\ar[dd,twoheadrightarrow] & & \frac{\aa\wedge\aa}{\aa\wedge\{\pm 1\}}\oplus\SS_2\ar[dd,equal]\\
   \hat{E}^2_{2,1}\ar[ru,twoheadrightarrow,"\hat{d}^2_{2,1}"]\ar[rr,"\alpha"]\ar[dd,twoheadrightarrow,"\delta''"] & & E^2_{2,1}\ar[ru,"d^2_{2,1}"]\ar[dd,twoheadrightarrow,"\delta"],\ar[rr,"\pi'"]\ar[ru]\ar[dd,twoheadrightarrow] & & \coker(\alpha)\ar[ru,"d'"]\ar[dd,twoheadrightarrow,"\delta'" near end] &\\
   & & & \frac{\aa\wedge\aa}{\aa\wedge\mu_2(A)}\oplus\SS_2\ar[rr,equal] & & \frac{\aa\wedge\aa}{\aa\wedge\mu_2(A)}\oplus\SS_2\\
   \II_A\psi_1(-1)\ar[rr,hook] & & \RP_1(A)\ar[rr,twoheadrightarrow,"\pi"]\ar[ru,"\Delta'"] & & \frac{\RP_1(A)}{\II_A\psi_1(-1)}\ar[ru,"\bar{\Delta'}"] &\\
\end{tikzcd}
}
\]
where $\pi$, $\pi'$ and $\pi''$ are the canonical maps. Note that $\im(\Delta'')\subseteq\aa\wedge\mu_2(A)$ because $\im(\gamma)\subseteq\im(\alpha)$; another consequence of this inclusion is that $\delta'$ is an isomorphism, and using the diagram above, one checks that the map
\[
    \bar{\Delta'}:\frac{\RP_1(A)}{\II_A\psi_1(-1)}\to\frac{\aa\wedge\aa}{\aa\wedge\mu_2(A)}\oplus\SS_2
\]
induced by $\Delta'$ coincides with $d'\circ(\delta')^{-1}:\frac{\RP_1(A)}{\II_A\psi_1(-1)}\to\frac{\aa\wedge\aa}{\aa\wedge\mu_2(A)}\oplus\SS_2$. Its kernel is exactly $\frac{\RB(A)}{\II_A\psi_1(-1)}$ by Lemma \ref{ker-rba}, and this group is isomorphic to $\widetilde{\coker(\alpha)}$ via $\delta'$.
The second assertion follows from Lemma \ref{thm-main} (3).
\end{proof}
\begin{exm}
    \begin{enumerate}
        \item Coronado and Hutchinson studied the cases $A=\z$ and $A=\z\half$ in \cite[\S 8]{C-H2022}; for the complete case of $A=\z\half$ see Theorem \ref{thm-sm2-sl2-g}. 
        \item For the ring $A=\z\mth{10}$ we have the exact sequence
        \[
            H_3(\SM_2(\z\mth{10}),\z)\to H_3(\SL_2(\z\mth{10}),\z)\to \RB(\z\mth{10})\to 0.
        \]
        \item A finite field $\F_q$ with $q=p^d$ and $p>2$ satisfies $-1\notin \aa^2$ precisely when $q\equiv 3 \pmod 4$. In this case, if $(p-1)d>6$, we have the exact sequence
        \[
            H_3(\SM_2(\F_q),\z)\to H_3(\SL_2(\F_q),\z)\to \RB(\F_q)\to 0.
        \]
        \item For the finite local rings $\z/p^r$ with $p>9$ and $-1\notin\aa^2$, we have the exact sequence
        \[
            H_3(\SM_2(\z/p^r),\z)\to H_3(\SL_2(\z/p^r),\z)\to \RB(\z/p^r)\to 0.
        \]
        A more complete result for finite local rings such as $\z/p^r$ was obtained by Mirzaii and Rojas in the preprint \cite[Theorem 6.2]{B-R--2025}.
    \end{enumerate}
\end{exm}
\begin{rem}\label{lack-on-zm}
    The only obstruction to producing further examples of the form $A_m=\z\mth{m}$ seems to lie in the condition 
    \[
        H_3(B(A),\z)\simeq H_3(T(A),\z).    
    \]
    
    In Proposition \ref{prp-sm2-sl2}, this condition is used to guarantee that $E^\infty_{0,3}$ is a quotient of $H_3(T(A),\z)$. An alternative way of ensuring this is to require that the summand $\SS_3$ of $H_3(B(A),\z)$ be contained in $\im(d^2_{2,2})$.
    
    Along these lines, Proposition \ref{prp-hb-ht-am} (3) suggests further examples of rings $\z\mth{m}$ with $r=2$. For instance, the ring $\z\mth{290}$ satisfies all the conditions except the isomorphism above, and all that is needed is an element $x\in E^2_{2,2}$ such that $d^2_{2,2}(x)$ generates $\SS_3\simeq\z/3$. A similar situation occurs for the ring of Eisenstein integers $\z[\zeta]$. In fact, an analysis of the Lyndon/Hochschild-Serre spectral sequence of the extension
    \[
        1\to N(\z[\zeta]) \to B(\z[\zeta])\to T(\z[\zeta])\simeq\lan\zeta\ran\to 1
    \]
    shows that $\SS_3$ fits into an exact sequence
    \[
        0\to \z/6\to\SS_3\to \z/3\to 0.
    \]
\end{rem}
The next result generalizes Theorem 6.6 of \cite{B-E--2023} by removing the condition $H_2(B(A),\z)\simeq H_2(T(A),\z)$. 
\begin{cor}
    \label{thm-main-ant}
    Let $A$ be a ring satisfying the following conditions:
    \begin{enumerate}
        \item $\mu_2(A)=\{\pm 1\}$ and $-1\in\aa^2$;
        \item the complex $X_\bullet(A^2)\to\z$ is exact in dimension $<2$;
        \item $H_3(B(A),\z)\simeq H_3(T(A),\z)$.
    \end{enumerate}
    Then there exists an exact sequence
    \[
        H_3(\SM_2(A),\z)\to H_3(\SL_2(A),\z)\to \RB(A)\to 0.
    \]
\end{cor}
\begin{proof}
    If $\mu_2(A)=\{\pm 1\}$ and $-1\in\aa^2$, then $\im(\gamma)=\im(\alpha)$ by \cite[Corollary 6.4]{B-E--2023}, and we may apply the second part of Theorem \ref{thm-main-1}.
\end{proof}
In what follows, we consider domains. Let $A$ be a domain with $\char(A)\neq 2$. Recall the exact sequence \eqref{d-gam-seq}
\[
    \begin{tikzcd}
        \II_A\otimes\mu_2(A)\ar[r,"\gamma"] & E^2_{2,1}\ar[r,"\delta"] & \RP_1(A) \ar[r] & 0.
    \end{tikzcd}
\]
If $\II_A^2\otimes\mu_2(A)\subseteq\ker(\gamma)$ and $H_2(B(A),\z)\simeq H_2(T(A),\z)$ (i.e., $\SS_2=0$), then the differential $d^2_{2,1}$ yields the commutative diagram
\[
    \begin{tikzcd}
        &\II_A\otimes\mu_2(A)\ar[r,"\gamma"]\ar[d] & E^2_{2,1}\ar[r,"\delta"]\ar[d,"d^2_{2,1}"] & \RP_1(A) \ar[r]\ar[d] & 0\\
        0 \ar[r] & \aa\wedge\mu_2(A) \ar[r]& \aa\wedge\aa \ar[r] & \frac{\aa\wedge\aa}{\aa\wedge\mu_2(A)} \ar[r] & 0
    \end{tikzcd}
\]
and the snake lemma yields the exact sequence
\[
    \begin{tikzcd}
        \mathcal{K}\ar[r] & E^{\infty}_{2,1}\ar[r] & \RB(A)\ar[r] & 0,
    \end{tikzcd}
\]
where $\mathcal{K}$ is the kernel of the left vertical map.
Note that, by the proof of Corollary \ref{ker-gam}, every element of $\mathcal{K}$ is of the form $K+\Lan\nu\Ran\otimes(-1)$, where $K\in\II^2_A\otimes\mu_2(A)$ and $\nu$ is an $l$-th root of unity with $l$ even; moreover, note that when $-1\notin\aa^2$ we can take $\nu\in\mu_2(A)=\{\pm 1\}$ as in the proof of Lemma \ref{thm-main} (5). If $\II^2_A\otimes\mu_2(A)\subseteq\ker(\gamma)$, we obtain the exact sequence
\[
    \begin{tikzcd}
        \lan\Lan-1\Ran\otimes(-1)\ran\ar[r,"\gamma"] & E^{\infty}_{2,1}\ar[r,"\delta"] & \RB(A)\ar[r] & 0,
    \end{tikzcd}
\]
and in the spectral sequence $\hat{E}^1_{p,q}\implies H_3(\SM_2(A),\z)$ the term $\hat{E}^{\infty}_{2,1}$ is generated by $\lan-1\ran$. We thus obtain the diagram
\begin{equation}
\begin{tikzcd}
& \hat{E}^{\infty}_{2,1}\ar[d,"\alpha:=\inc_*"] & &\\
\lan\Lan-1\Ran\otimes (-1)\ran\ar[r,"\gamma"]& E^{\infty}_{2,1} \ar[r,"\delta"] & \RB(A)\ar[r] & 0,
\end{tikzcd}
\end{equation}
in which $\im(\gamma)=\im(\alpha)$ by Lemma \ref{alpha-gam-chi} (6); this gives the exact sequence
\[
    \begin{tikzcd}
        \hat{E}^{\infty}_{2,1}\ar[r,"\alpha"]& E^{\infty}_{2,1} \ar[r,"\delta"] & \RB(A)\ar[r] & 0.
    \end{tikzcd}
\]
Applying the same reasoning as in the proof of Proposition \ref{prp-sm2-sl2}, we obtain our second main theorem.
\begin{thm}
    \label{prp-sm2-sl2-dom}
    Let $A$ be a domain with $\char(A)\neq2$ such that
    \begin{enumerate}
        \item the complex $X_{\bullet}(A^2)\to\z$ is exact in dimension $<2$;
        \item $H_i(B(A),\z)\simeq H_i(T(A),\z)$ for $i=2,3$;
        \item $-1\notin\aa^2$;
        \item $\II^2_{A}\otimes\mu_2(A)\subseteq\ker(\gamma)$.
    \end{enumerate}
    Then there is an exact sequence
    \[
        H_3(\SM_2(A),\z)\arr H_3(\SL_2(A),\z)\arr \RB(A) \arr 0.
    \]
\end{thm}
As an application to the domains $\z\mth{m}$, we obtain the following corollary.
\begin{cor}\label{sm-sl-zm}
    Let $r\geq 1$ and $m=2^s\cdot 3^tp_1^{t_1}\cdots p_r^{t_r}$ where $s,t,t_i\geq 1$. Suppose that the ring $A=\z\mth{m}$ satisfies the following conditions:
    \begin{enumerate}
        \item for each $i\leq r$, $\left(\frac{\z}{p_i\z}\right)^\times$ is generated by $\{-1,2,3,p_1,\dots,p_{i-1}\}$ (Morita's condition);
        \item for each $i\leq r$, $p_i=a^{2l}+b^{2n}$, where $a,b\in\{-1,2,3,p_1,\dots,p_{i-1}\}$ and $l,n\geq 1$.
    \end{enumerate}
    Then there is an exact sequence
    \[
        \begin{tikzcd}
            H_3(\SM_2(A),\z)\arr H_3(\SL_2(A),\z)\arr \RB(A) \arr 0.
        \end{tikzcd}
    \]
\end{cor}
\begin{proof}
    First, note that $A$ is a universal $\GE_2$-ring by Example \ref{ex-ge2} (4), so it satisfies condition (1) of Theorem \ref{prp-sm2-sl2-dom}; moreover, it satisfies condition (2) by Lemma \ref{ht-hb-l1}, since $6$ is a unit. For condition (4), note that the group $\GG_A$ is generated by $\{\lan-1\ran,\lan 2\ran,\lan 3\ran,\lan p_1\ran,\dots,\lan p_r\ran\}$, and that $\gamma(\Lan x\Ran\Lan y\Ran\otimes(-1))=\Lan x\Ran\Lan y\Ran\chi_1$ for any units $x,y$; note also that $\Lan x\Ran^2\otimes(-1)=0$ for any unit $x$. Now, any product $\Lan x\Ran\Lan y\Ran$ with $\lan x\ran\neq\lan y\ran$ can be expanded as a sum of elements of $\II_A^2$, each containing a factor $\Lan-1\Ran$ or $\Lan z\Ran$, where $z$ is a sum of two squares; such elements annihilate $\chi_1$ by Lemma \ref{alpha-gam-chi} (7) and Remark \ref{sum-2-sqr-chi}. This implies (4). Condition (3) of Theorem \ref{prp-sm2-sl2-dom} holds trivially.
\end{proof}
\section{A refined Bloch-Wigner sequence over local fields}
The aim of this section is to obtain a refined Bloch-Wigner exact sequence for a (non-archimedean) local field $F$ with $\char(F)\neq 2$. We further assume that $F$ is non-dyadic, i.e., that $\char(\bar{F})\neq 2$, where $\bar{F}$ denotes the residue field. Let $\mathcal{O}_F$ be the ring of integers of $F$ and let $U=\mathcal{O}_F^\times$. It is well known that $\GG_F=\{1,\lan u\ran,\lan \pi\ran,\lan u\pi\ran\}$, where $\pi$ is a {\it uniformizer} and $u\in U$ is a unit whose reduction is not a square in $\bar{F}$. Since $\mu(F)$ is finite cyclic, let $\zeta$ be a generator and set $2N=|\mu(F)|$. We may take $u=\zeta$, or $u=-1$ if $-1$ is not a square in $F$, in which case $N$ is odd.  
\begin{cor}\label{sm2-sl2-seq-local}
    Let $F$ be a non-dyadic local field with $\char(F)\neq 2$. Then there is an exact sequence
    \[
        \begin{tikzcd}
            H_3(\SM_2(F),\z)\arr H_3(\SL_2(F),\z)\arr \RB(F) \arr 0.
        \end{tikzcd}
    \]
\end{cor}
\begin{proof}
    We consider the case where $-1\notin {F^{\times}}^2$ (the case in which $-1$ is a square follows from \cite[Corollary 6.4]{B-E--2023}), which is condition (3) of Theorem \ref{prp-sm2-sl2-dom}. Conditions (1) and (2) hold because such fields are infinite. As for condition (4), take any product $\Lan a\Ran\Lan b\Ran\otimes(-1)$; the only relevant case is $a=-1$ and $b=\pi$, where $\pi$ is a uniformizer. In this case, Lemma \ref{alpha-gam-chi} (3) and (7) give
    \[
        \gamma(\Lan -1\Ran\Lan \pi\Ran\otimes(-1))=\Lan \pi\Ran\Lan -1\Ran\chi_1=0.
    \]
    This implies (4).
\end{proof}
\begin{rem}
    \label{lf-classes}
    Note that in any non-dyadic local field $F$ we have $F^\times\wedge F^\times=\lan\pi\wedge\zeta\ran+2(F^\times\wedge F^\times)$, where $\zeta$ is a generator of $\mu(F)$. This follows by replacing $a$ and $b$ with representatives of their square classes in any element $a\wedge b$. Consequently, if $-1\notin {F^\times}^2$, we may replace $\zeta$ by $-1$. In general, if $A$ is a domain such that $\GG_A=\{1,\lan -1\ran,\lan a\ran,\lan -a\ran\}$ we have $A^\times\wedge A^\times=\lan a\wedge(-1)\ran+2(A^\times\wedge A^\times)$.
\end{rem}

We now analyze $H_3(\SM_2(F),\z)$ for a local field $F$ with $-1\notin {F^\times}^2$, using the spectral sequence $\hat{E}^1_{p,q}\implies H_3(\SM_2(F),\z)$. By \cite[Lemma 5.1]{B-E--2023}, $\displaystyle \hat{E}^1_{2,2}=H_2(\SM_2(F),\z)\simeq\frac{F^\times\wedge F^\times}{F^\times\wedge(-1)}$, and the map $\hat{d}^1_{2,2}$, which comes from the transfer map, satisfies $\overline{a\wedge b}\mapsto 2(a\wedge b)$. Lemma \ref{ker-wed-2} implies that $\hat{d}^1_{2,2}$ is injective, so that $\hat{E}^{\infty}_{2,2}=\hat{E}^2_{2,2}=0$. Since $\hat{d}^1_{1,2}=0$, the group $\hat{E}^{\infty}_{1,2}=\hat{E}^2_{1,2}$ is given by 
\[
    \frac{\hat{E}^1_{1,2}}{\im(\hat{d}^1_{2,2})}=\frac{F^\times\wedge F^\times}{2(F^\times\wedge F^\times)}\simeq\lan\overline{\pi\wedge-1}\ran.
\]

In order to determine $\hat{E}^2_{0,3}$, we need some information about the group $H_3(F^\times,\z)$.

Let $C$ be an abelian group and let $\sigma_1:\tors(C,C)\arr \tors(C,C)$
be the map obtained by interchanging the two copies of $C$. One verifies that $\sigma_1$ is induced by the involution $C\otimes C \arr C\otimes C$, $a\otimes b\mapsto -b\otimes a$.

Let $\Sigma_2'=\{1, \sigma'\}$ be the symmetric group of order $2$, and consider the following action of 
$\Sigma_2'$ on $\tors(C,C)$:
\[
(\sigma', x)\mapsto -\sigma_1(x).
\]

\begin{prp}\label{H3B}
For any abelian group $C$ there is an exact sequence
\[
\begin{array}{c}
0 \arr \bigwedge_\z^3 C \arr H_3(C,\z) \arr \tors(C,C)^{\Sigma_2'} \arr 0,
\end{array}
\]
where the homomorphism on the right is obtained from the composite
\[
H_3(C,\z) \overset{{\Delta_C}_\ast}{\larr } H_3(C\oplus C,\z) \arr \tors(C,C),
\]
$\Delta_C$ being the diagonal map $C \arr C\oplus C$, $x \mapsto (x,x)$.
\end{prp}
\begin{proof}
See \cite[Lemma~5.5]{suslin1991}, \cite[Section 6]{breen1999}.
\end{proof}

When $C=F^\times$ we have $\tors(F^\times,F^\times)^{\Sigma'_2}=\tors(\mu(F),\mu(F))=\mu(F)$; the first equality holds because $F$ is a domain and $\mu(F)$ is the torsion of $F^\times$. The map $\hat{d}^1_{1,3}$ induces the commutative diagram
\[
    \begin{tikzcd}
        0 \ar[r] & \bigwedge^3F^\times \ar[r]\ar[d,"2"] & H_3(F^\times,\z)\ar[r]\ar[d,"\hat{d}^1_{1,3}"] & \mu(F)\ar[r]\ar[d,"0"] & 0\\
        0 \ar[r] & \bigwedge^3F^\times \ar[r] & H_3(F^\times,\z)\ar[r] & \mu(F)\ar[r] & 0.
    \end{tikzcd}
\]
Indeed, $\hat{d}^1_{1,3}=w_*-\id_*$, where $w_*$ is induced by conjugation by the matrix $w$; this action sends $a\mapsto a^{-1}$ for every $a\in F^\times$, which implies that the left-hand map is multiplication by $2$. On the other hand, any generator $\zeta$ of the finite cyclic group $\mu(F)$ has a preimage $\Theta_{\zeta}$ represented by the cycle (where $2N=|\mu(F)|$)
\[
    \Theta_{\zeta}=\sum^{2N-1}_{i=0}[\zeta|\zeta^i|\zeta].
\]
The element $w\Theta_{\zeta}-\Theta_{\zeta}$ vanishes in homology; indeed, it is the boundary of
\[
    \sum^{2N-1}_{i=0}\left\{[w|\zeta|\zeta^i|\zeta]-[\zeta^{-1}|w|\zeta^i|\zeta]+[\zeta^{-1}|\zeta^{-i}|w|\zeta]-[\zeta^{-1}|\zeta^{-i}|\zeta^{-1}|w]+[\zeta|\zeta^{-1}|\zeta^{-i}|\zeta^{-1}]-[\zeta|\zeta^i|\zeta|\zeta^{-1}]\right\}.
\]

This implies that the right-hand vertical map is $0$. The left-hand vertical map is surjective, since $\GG_F$ has only four elements: it suffices to pass to square classes in each factor of a generator and then use linearity. By the snake lemma we obtain $\hat{E}^2_{0,3}=\mu(F)$, and since $\hat{E}^2_{2,2}=0$, it follows that $\hat{E}^{\infty}_{0,3}=\mu(F)$. Moreover, $\hat{E}^2_{2,1}=\GG_F$ and $\hat{d}^2_{2,1}(\lan a\ran)=a\wedge(-1)$ (see \cite[\S 5]{B-E--2023}); hence, by Lemma \ref{ker-d-alpha}, $\hat{E}^3_{2,1}=\ker(\hat{d}^2_{2,1})$ is generated by $\lan\zeta\ran$, where $\zeta$ is a generator of $\mu(F)$, and since $-1$ is not a square in $F$, we have $\lan\zeta\ran=\lan-1\ran$. 

Consider now the filtration
\[
    0\subseteq \hat{F}_0\subseteq \hat{F}_1\subseteq \hat{F}_2= \hat{F}_3=H_3(\SM_2(F),\z).
\]
It yields two exact sequences
\begin{equation}
    \label{sq-sm2-1}
    0 \arr \hat{F}_0=\mu(F)\arr \hat{F}_1\arr \hat{E}^{\infty}_{1,2}\arr 0
\end{equation}
and
\begin{equation}
    \label{sq-sm2-2}
    0 \arr \hat{F}_1\arr H_3(\SM_2(F),\z)\arr \hat{E}^{\infty}_{2,1}\arr 0.
\end{equation}

In \eqref{sq-sm2-1}, a lift of $\overline{\pi\wedge\zeta}$ is the element
\begin{align*}
    \Pi_{\zeta}=&\bigg\{[w|\zeta|\pi]-[w|\pi|\zeta]-[\zeta^{-1}|w|\pi]+[\pi^{-1}|w|\zeta]+[\zeta^{-1}|\pi^{-1}|w]-[\pi^{-1}|\zeta^{-1}|w]+[\zeta|\zeta^{-1}|\pi^{-1}]\\
    &-[\pi|\pi^{-1}|\zeta^{-1}]-[\pi|\zeta|\zeta^{-1}\pi^{-1}]+[\zeta|\pi|\zeta^{-1}\pi^{-1}]\bigg\}\otimes 1.
\end{align*}
Its image in $H_3(\SM_2(F),\z)$ under the map in \eqref{sq-sm2-2} is a $2$-torsion element. This implies that $2\Pi_{\zeta}=0$, and hence that \eqref{sq-sm2-1} splits. Indeed,
\begin{align*}
    2\Pi_{\zeta}-c(-1,\pi,\zeta)=&d_4\big(\{[w|w|\zeta|\pi]-[w|w|\pi|\zeta]-[w|\zeta^{-1}|w|\pi]+[w|\pi^{-1}|w|\zeta]+[\zeta|w|w|\pi]-[\pi|w|w|\zeta]\\
    &+[w|\zeta^{-1}|\pi^{-1}|w]-[w|\pi^{-1}|\zeta^{-1}|w]-[\zeta|w|\pi^{-1}|w]+[\pi|w|\zeta^{-1}|w]+[\zeta|\pi|w|w]\\
    &-[\pi|\zeta|w|w]-[\zeta^{-1}|w|\zeta^{-1}|\pi^{-1}]+[\pi^{-1}|w|\pi^{-1}|\zeta^{-1}]+[w|\zeta|\zeta^{-1}|\pi^{-1}]\\
    &-[w|\pi|\pi^{-1}|\zeta^{-1}]+[\zeta^{-1}|\zeta|w|\pi^{-1}]-[\pi^{-1}|\pi|w|\zeta^{-1}]-[\zeta^{-1}|\zeta|\pi|w]+[\pi^{-1}|\pi|\zeta|w]\\
    &+[\pi^{-1}|w|\zeta|\zeta^{-1}\pi^{-1}]-[\zeta^{-1}|w|\pi|\pi^{-1}\zeta^{-1}]-[\pi^{-1}|\zeta^{-1}|w|\zeta^{-1}\pi^{-1}]\\
    &+[\zeta^{-1}|\pi^{-1}|w|\pi^{-1}\zeta^{-1}]-[w|\pi|\zeta|\zeta^{-1}\pi^{-1}]+[w|\zeta|\pi|\pi^{-1}\zeta^{-1}]\\
    &-[\zeta^{-1}|\pi^{-1}|\zeta\pi|w]+[\pi^{-1}|\zeta^{-1}|\pi\zeta|w]-[\zeta^{-1}|\pi^{-1}|\pi|\zeta]+[\pi^{-1}|\zeta^{-1}|\zeta|\pi]\\
    &-[\zeta|\zeta^{-1}\pi^{-1}|\pi|\zeta]+[\pi|\pi^{-1}\zeta^{-1}|\zeta|\pi]-[\pi|\zeta|\zeta^{-1}\pi^{-1}|\pi\zeta]+[\zeta|\pi|\pi^{-1}\zeta^{-1}|\zeta\pi]\\
    &+[\zeta|\zeta^{-1}|\pi^{-1}|\pi]-[\pi|\pi^{-1}|\zeta^{-1}|\zeta]-[\pi|\pi^{-1}|\pi|\zeta]+[\zeta|\zeta^{-1}|\zeta|\pi]\}\otimes 1\big),
\end{align*}
where $c(-1,\pi,\zeta)$ is the image of $-1\wedge\pi\wedge\zeta\in\bigwedge^3F^\times$, and this last element vanishes because $-1=\zeta^{|\mu(F)|/2}$. Note that in all the computations above we may replace $\zeta$ by $-1$ when $-1\notin{F^\times}^2$; thus the element $\Pi_{-1}$ is $2$-torsion as well.\\ 
\begin{prp}
    Let $F$ be a non-dyadic local field with $\char(F)\neq 2$. If $-1\notin{F^\times}^2$, then
    \[
        H_3(\SM_2(F),\z)\simeq \frac{\z}{(4N)\z}\oplus\frac{\z}{2\z}, 
    \]
    where $2N=|\mu(F)|$.
\end{prp}
\begin{proof}
    Since $-1$ is not a square, the exact sequence \eqref{sq-sm2-2} reads
    \[
        \begin{tikzcd}
            0 \ar[r]& \mu(F)\oplus\lan\overline{\pi\wedge(-1)}\ran\ar[r] & H_3(\SM_2(F),\z)\ar[r] & \hat{E}^{\infty}_{2,1}=\lan\lan -1\ran\ran\ar[r] & 0.
        \end{tikzcd}
    \]
    Taking a lift of $\lan -1\ran$ on the right-hand side, we obtain the cycle
    \[
        \Gamma = \{[w|w|-1]-[w|-1|w]+[w|-1|-1]+[-1|w|w]-[-1|w|-1]+[-1|-1|w]\}\otimes 1.
    \]
    A direct calculation gives that
    \begin{align*}
        2\Gamma-[-1|-1|-1]\otimes 1=&d\big(\{[w|w|w|-1]+[-1|-1|w|w]-[-1|w|-1|w]+[w|-1|-1|w]\\
        &+[w|w|-1|-w]-[w|-1|w|-w]+[-1|w|w|-w]\}\otimes 1\big).
    \end{align*}
    Recall the element
    \[
        \Theta_{\zeta}=\sum_{i=0}^{2N-1}[\zeta|\zeta^i|\zeta]
    \]
    lifting the generator $\zeta\in\mu(F)$; by abuse of notation we also write $\Theta_\zeta$ for its image in $H_3(\SM_2(F),\z)$. Note that the element $\Theta_{\zeta}+\Gamma$ generates a cyclic subgroup of order $4N$ since $N$ is odd, while $\Pi_{-1}$ generates a subgroup of order $2$; this implies that
    \[
        H_3(\SM_2(F),\z)\simeq\frac{\z}{(4N)\z}\oplus\frac{\z}{2\z}.
    \]
\end{proof}

Note that $\mu(F)^{\tilde{}}\simeq\frac{\z}{(4N)\z}$ is the unique non-trivial extension of $\mu(F)$ by $\z/2$. We now return to the spectral sequence $E^1_{p,q}\implies H_{p+q}(\SL_2(F),\z)$.
\begin{lem}
    \label{d122-surj}
    If $-1\notin {F^{\times}}^2$, then $d^1_{2,2}$ is surjective.
\end{lem}
\begin{proof}
    For such fields we know that $\GG_F=\{1,\lan u\ran,\lan\pi\ran,\lan u\pi\ran\}$, where the reduction of $u$ is a non-square in the residue field.
    Since $-1\notin {F^\times}^2$, we may take $u=-1$; moreover, $\pi+1=r^2$ for some $r$ (by Hensel's lemma). Consider now the long exact sequence \eqref{long-eseq} and the element $[-1]\otimes\partial_3((\pmb{\infty},\pmb{0},\pmb{1},\pmb{\pi+1}))$, which represents an element of $H_1(\SL_2(F),Z_2)$. This element maps to $\Lan\pi+1\Ran\Lan-\pi\Ran\otimes(-1)\in\RR_F\otimes\mu_2(F)$, which is zero. Hence there exists an element $l(\pi)\in H_2(\SL_2(F),Z_1)$ which is represented in the bar resolution by
    \[
        l(\pi)=R_{\pi}\otimes\partial_2(X_{-\pi^{-1}})+H_{\pi}\otimes\partial_2(X_1),
    \]
    where
    \[
        R_{\pi}=[g_1^{-1}|-1]-[-1|g_1^{-1}]-[h_1^{-1}r|-1]+[-1|h_1^{-1}r],\quad H_{\pi}=[r|-1]-[-1|r].
    \]
    Here $h_{a}=\mtxx{1}{a^{-1}}{0}{1}$ and $g_a=\mtxx{0}{1}{-1}{a}$ for $a\in F^{\times}$. We shall show that $d^1_{2,2}(l(\pi))=-1\wedge\pi\in H_2(T(F),\z)\simeq H_2(\SL_2(F),X_1)$ in the spectral sequence $E^1_{p,q}$. This image is represented by
    \[
        ((g_{-\pi^{-1}}^{-1}-h_{-\pi}^{-1}+1)R_{\pi}+(g_1^{-1}-h_1^{-1}+1)H_{\pi})\otimes(\pmb{\infty},\pmb{0}),
    \]
    which is an element of $H_2(\SL_2(F),X_1)\simeq H_2(T(F),\z)$. To compute it, we use the section trick; to this end, we first rewrite the element above in the standard resolution: 
    \begin{align*}
        (g_{-\pi^{-1}}^{-1}-h_{-\pi}^{-1}+1)&((-1,g_1^{-1},-g_1^{-1})-(1,-1,-g_1^{-1})-(1,h_1^{-1}r,-h_1^{-1}r)+(1,-1,-h_1^{-1}r))\otimes(\pmb{\infty},\pmb{0})\\
        +(g_1^{-1}-h_1^{-1}+1)&((1,r,-r)-(1,-1,-r))\otimes(\pmb{\infty},\pmb{0}).
    \end{align*}
    Take the section $s:T(F)\backslash\SL_2(F)\to \SL_2(F)$ of the canonical projection $p:\SL_2(F)\to T(F)\backslash\SL_2(F)$ given by
    \[
        s(T(F)\mtxx{a}{b}{c}{d}):=\begin{cases}
            \mtxx{1}{ba^{-1}}{ca}{da}, & a\neq 0,\\
            \mtxx{0}{1}{-1}{bd}, & a = 0.
        \end{cases}
    \]
    We define a map $g\mapsto\overline{g}:=g(s\circ p(g))^{-1}$, which allows us to pass from the standard $\SL_2(F)$-resolution to the standard $T(F)$-resolution. Applying this map and simplifying, we obtain
    \begin{align*}
        ((-\pi^{-1},-(\pi^{-1}+1),\pi^{-1}+1)&-(-\pi^{-1},\pi^{-1},\pi^{-1}+1)\\
        -(1,\pi^{-1}+1,-(\pi^{-1}+1))&+(1,-1,-(\pi^{-1}+1)))\otimes(\pmb{\infty},\pmb{0}),
    \end{align*}
    which, in the bar resolution, becomes $([1+\pi|-1]-[-1|1+\pi]-[\pi^{-1}+1|-1]+[-1|\pi^{-1}+1])\otimes(\pmb{\infty},\pmb{0})$; this represents the element $(1+\pi)\wedge(-1)-(\pi^{-1}+1)\wedge(-1)=\pi\wedge(-1)$ in $H_2(\SL_2(F),X_1)\simeq H_2(T(F),\z)\simeq\bigwedge^2F^{\times}$.\\

    Finally, Remark \ref{lf-classes} shows that every $a\wedge b\in\bigwedge^2 F^{\times}$ can be written as $a\wedge b=n\{\pi\wedge(-1)\}+2K$ for some integer $n$ and some $K\in\bigwedge^2F^{\times}$, and it is straightforward to check that the element $([u|v]-[v|u])\otimes\{(\pmb{\infty},\pmb{0})+(\pmb{0},\pmb{\infty})\}$, which represents an element of $E^1_{2,2}$, maps to $2(u\wedge v)$. This completes the proof.
\end{proof} 
\begin{rem}
    \label{rem-d211-gen}
    More generally, let $A$ be a universal $\GE_2$-ring and consider the long exact sequence \eqref{long-eseq}. For $a\in\WW_A$, the cycle
    \[
        p_{-1}^+[-1]\otimes\partial_3(\pmb{\infty},\pmb{0},\pmb{1},\pmb{a})+\Lan1-a\Ran[-1]\otimes\partial_3((\pmb{\infty},\pmb{0},\pmb{1},\pmb{a})+(\pmb{0},\pmb{\infty},\pmb{1},\pmb{a})),
    \]
    where $p_{-1}^+=\lan -1\ran+1$, represents an element $\tilde{g}(a)\in H_1(\SL_2(A),Z_2)$ which maps to zero in $H_1(\SL_2(A),X_2)\simeq\RR_A\otimes\mu_2(A)$. Hence there exists an element
    \begin{align*}
        g(a)=&([g_1^{-1}|-1]-[-1|g_1^{-1}]-[1-a|-1]+[-1|1-a])\otimes X_{(1-a)^{-1}}\\
        &-([h_1^{-1}|-1]-[-1|h_1^{-1}]-[1-a|-1]+[-1|1-a])\otimes X_{a(1-a^{-1})}\\
        &+([g_{-1}^{-1}|-1]-[-1|g_{-1}^{-1}]-[w|-1]+[-1|w])\otimes X_{-(1-a)^{-1}}\\
        &-([h_{-1}^{-1}|-1]-[-1|h_{-1}^{-1}]-[wa|-1]+[-1|wa])\otimes X_{-a(1-a)^{-1}}\\
        &-([wa|-1]-[-1|wa])\otimes X_{-a}+([w|-1]-[-1|w])\otimes X_{-1}
    \end{align*}
    in $H_2(\SL_2(A),Z_1)$. The same reasoning, together with the section trick used in the proof above (with the same map $s$), shows that $d^1_{2,2}(g(a))=-1\wedge (1-a)$. Setting $a=1-\pi$ recovers the result above, but the argument also applies to other universal $\GE_2$-rings $A$ with $|\GG_A|=4$ in which $-1$ is not a square. For example, in $\z\half$ we have that $d^1_{2,2}(g(-1))=-1\wedge 2$.  
\end{rem}
The lemma above implies that $E^{\infty}_{1,2}=E^2_{1,2}=0$. The spectral sequence yields a filtration
\[
    0\subseteq F_0 = F_1\subseteq F_2 = F_3=H_3(\SL_2(F),\z),
\]
and comparison with the spectral sequence for $\SM_2(F)$ gives the commutative diagram
\[
    \begin{tikzcd}
        0\ar[r] & \mu(F)\ar[r]\ar[d] & \hat{F}_1 \ar[r]\ar[d] & \hat{E}^{\infty}_{1,2}=\lan\overline{\pi\wedge(-1)}\ran\ar[r]\ar[d] & 0\\
        0\ar[r] & \frac{\mu(F)}{\im(d^2_{2,2})}\ar[r] & F_1 \ar[r] & E^{\infty}_{1,2}=0\ar[r,equal] & 0
    \end{tikzcd}.
\]

The upper row is split by $\rho(\overline{\pi\wedge(-1)})=\Pi_{-1}$, which yields the following commutative diagram: 
\[
    \begin{tikzcd}
        \lan\overline{\pi\wedge(-1)}\ran\ar[r]\ar[d] & H_3(\SM_2(F),\z)\ar[d]\\
        0\ar[r] & H_3(\SL_2(F),\z).
    \end{tikzcd}
\]
Hence the image of $\Pi_{-1}$ under $H_3(\SM_2(F),\z)\to H_3(\SL_2(F),\z)$ is trivial, and the exact sequence \eqref{sm2-sl2-seq} (over the field $F$) yields the exact sequence
\[
    \begin{tikzcd}
        \mu(F)^{\tilde{}}\ar[r] & H_3(\SL_2(F),\z)\ar[r] & \RB(F)\ar[r] & 0.
    \end{tikzcd}
\]

Finally, from the classical Bloch-Wigner exact sequence over the algebraic closure $F^\alg$ of $F$, we obtain the commutative diagram
\[
    \begin{tikzcd}
        &\mu(F)^{\tilde{}}\ar[r]\ar[d,hook] & H_3(\SL_2(F),\z)\ar[r]\ar[d] & \RB(F)\ar[r]\ar[d] & 0\\
        0 \ar[r] & \mu(F^\alg)^{\tilde{}}\ar[r] & K^\ind_3(F^\alg)\simeq H_3(\SL_2(F^\alg),\z)\ar[r] & \BB(F^\alg)\ar[r] & 0
    \end{tikzcd}.
\]

The injectivity of the left-hand vertical map implies the injectivity of the left-hand map in the upper row, and we finally obtain the following theorem.
\begin{thm}\label{thm-sm2-sl2-local}
    Let $F$ be a non-dyadic local field with $\char(F)\neq 2$ and $-1\notin {F^\times}^2$. Then there is a refined Bloch-Wigner exact sequence
    \[
        0\to\mu(F)^{\tilde{}}\to H_3(\SL_2(F),\z)\to\RB(F)\to 0,
    \]
    where $\mu(F)^{\tilde{}}$ is the unique non-trivial extension of $\mu(F)$ by $\z/2$.
\end{thm}
All the reasoning in this section applies to any universal $\GE_2$-domain $A$ with finite $\mu(A)$ of order $2N$ and $\GG_A=\{1,\lan-1\ran,\lan a\ran,\lan-a\ran\}$, where $1-a$ is a unit (for example $\z\half$). Every step carries over, from Corollary \ref{sm2-sl2-seq-local} (since $\II_A^2\otimes\mu_2(A)\subseteq\ker(\gamma)$) to the injectivity of the left-hand map (passing to the algebraic closure $\ffrac(A)^\alg$ of the field of fractions of $A$), including the structure of $H_3(\SM_2(A),\z)$ (obtained by the same analysis of the spectral sequence),
\[
    H_3(\SM_2(A),\z)\simeq\frac{\z}{(4N)\z}\oplus\frac{\z}{2\z},
\]
while the surjectivity of $d^1_{2,2}$ follows from Remark \ref{rem-d211-gen} using the element $g(1-a)$, which satisfies $d^1_{2,2}(g(1-a))=-1\wedge a$. We can therefore state the following theorem.
\begin{thm}\label{thm-sm2-sl2-g}
    Let $A$ be a universal $\GE_2$-domain with $\char(A)\neq 2$ satisfying the following conditions:
    \begin{enumerate}
        \item $\mu(A)$ is finite;
        \item $H_i(B(A),\z)\simeq H_i(T(A),\z)$ for $i=2,3$;
        \item $\GG_A=\{1,\lan-1\ran,\lan a\ran,\lan -a\ran\}$, where $a\in\WW_A$.
    \end{enumerate}
    Then there is a refined Bloch-Wigner exact sequence
    \[
        0\to\mu(A)^{\tilde{}}\to H_3(\SL_2(A),\z)\to\RB(A)\to 0.
    \]
\end{thm}
Applying this theorem to the domain $\z\half$, we recover the results of Coronado and Hutchinson for this ring (see \cite[\S 8]{C-H2022}).

\end{document}